\documentclass[12pt]{article} 

\usepackage[margin=3cm]{geometry}
\usepackage[english]{babel}
\usepackage[hyperfootnotes=false]{hyperref}
\usepackage{amsmath}
\usepackage{amsthm}
\usepackage{amssymb}
\usepackage{ucs} 
\usepackage{enumerate}
\usepackage[utf8x]{inputenc}
\usepackage[T1]{fontenc}

  \newcommand{\N}{\mathbb N}
  
  \newcommand{\Z}{\mathbb Z}
  
  \newcommand{\LL}{\mathrm L}

  \newcommand{\inv}{^{-1}}
  \newcommand{\B}{\mathcal B}
  \renewcommand{\H}{\mathcal H}
  \newcommand{\BH}{\B(\H)}

  \newcommand{\eps}{\varepsilon}
  \newcommand{\ph}{\varphi}

\renewcommand{\circle}{{\mathbb S^1}}
  
  \renewcommand{\d}{\mathrm{d}}
  \renewcommand{\leq}{\leqslant}
  \renewcommand{\geq}{\geqslant}
  \newcommand{\abs}[1]{\left\lvert #1\right\rvert}

  \newcommand{\act}{\curvearrowright}

  \newcommand{\impl}{\Rightarrow}

  \newcommand{\la}{\left\langle}
  \newcommand{\ra}{\right\rangle}
  \newcommand{\into}{\hookrightarrow}

\newcommand{\free}{\mathrm F}
\newcommand{\covol}{\mathrm{covol}}

\usepackage{pgf,tikz}
\usetikzlibrary{arrows}
\usetikzlibrary{decorations.markings}

\DeclareMathOperator{\Aut}{\mathrm{Aut}}
\DeclareMathOperator{\RR}{\mathcal{R}}

\newtheorem{thmi}{Theorem}

\newtheorem{thm}{Theorem}[section]
\newtheorem{cor}[thm]{Corollary}
\newtheorem{lem}[thm]{Lemma}
\newtheorem{prop}[thm]{Proposition}

\theoremstyle{definition}

\newtheorem{df}[thm]{Definition}
\newtheorem{rmq}[thm]{Remark}
\newtheorem*{rmqi}{Remark}
\newtheorem{exemple}[thm]{Example}

\title{Orbit full groups for locally compact groups}

\author{A. Carderi and F. Le Ma\^itre}

\begin{document}

\maketitle

\begin{abstract}
  We show that the topological rank of an orbit full group generated by an ergodic, probability measure-preserving free action of a non-discrete unimodular locally compact Polish group is two. For this, we use the existence of a cross section and show that for a locally compact Polish group, the full group generated by any dense subgroup is dense in the orbit full group of the action of the group.

  We prove that the orbit full group of a free action of a locally compact Polish group is extremely amenable if and only if the acting group is amenable, using the fact that the full group generates the von Neumann algebra of the action. 
\end{abstract}

\tableofcontents

\section*{Introduction}

The full group of a measure-preserving action of a countable group on a standard probability space $(X,\mu)$ is the group of measure preserving transformations which preserve every orbit. It is a (complete) invariant of orbit equivalence for the action and has a natural Polish topology induced by the uniform metric $d_u(S,T):=\mu(\{x\in X: S(x)\neq T(x)\})$. This group topology encodes many interesting property of the action. For example Giordano and Pestov proved in \cite{MR2311665} that if the group acts freely, then the full group is extremely amenable if and only the acting group is amenable. Another example was provided by the second named author in \cite{gentopergo}: the cost of the action is very closely related to the topological rank of the full group, that is the minimum number of generator needed to generate a dense subgroup. 

Now let $G$ be a Polish group and consider a measure-preserving $G$-action on a standard probability space $(X,\mu)$. In \cite{Carderi:2014qr}, we  initiated the study of a Polish group topology on the associated \textit{orbit full group} $$[\RR_G]=\{T\in\Aut(X,\mu): \forall x\in X, T(x)\in G\cdot x\}$$ which coincides with the uniform topology when $G$ is countable discrete group. In this work, we want to analyse orbit full groups associated to free actions of second countable locally compact groups, which most of the time we will suppose unimodular. 

This study is motivated by the fact that for actions of locally compact groups, orbit full groups are still complete invariants of orbit equivalence \cite[Thm. 3.26]{Carderi:2014qr}, so their topological properties should reflect properties of the associated equivalence relation. Moreover, these orbit full groups are better behaved since they arise naturally as unitary groups (see Sec. \ref{sec:full group unitary group}) and ‘‘preserve density'' as follows.

\begin{thmi}[{see Thm. \ref{thm:densesub}}]\label{thmi:density}
  For every measure-preserving action of a locally compact Polish group $G$ on a probability space $(X,\mu)$ and for every dense subgroup $H\subset G$, the orbit full group $[\RR_H]$ is dense in $[\RR_G]$. 
\end{thmi}

The above theorem is false for general actions of Polish groups: an example of Kolmogorov gives a measure-preserving action of the bijection group  of the integers  $\mathfrak S_\infty$ such that whenever $H\leq \mathfrak S_\infty$ is a dense countable group, the full group $[\RR_H]$ is  \textit{not} dense in $[\RR_{\mathfrak S_\infty}]$ (see \cite[Ex. 3.14]{Carderi:2014qr}). 

We will actually prove that Theorem \ref{thmi:density} holds for every \textit{suitable} action of a Polish group (see Definition \ref{df:suitable} and Theorem \ref{thm:densesubpol})). This notion was introduced by Becker in \cite{MR2995370}, where he proved that any measure-preserving action of a locally compact Polish group is suitable. On the other hand, general Polish groups can have actions which are suitable and actions which are not. As an example, the standard Bernoulli shift of $\mathfrak S_\infty$ on $[0,1]^\N$ is suitable (see Example \ref{example: suitable bernoulli}).\\ 




Using Theorem \ref{thmi:density}, we can then show that orbit full groups associated to actions of locally compact, non-compact and non-discrete Polish groups contain a dense $2$-generated subgroup. This is in sharp contrast with the discrete case where the topological rank reflects the cost of the equivalence relation and thus can be equal to any integer $n\geq 2$. Our result also shows that cost cannot provide a rich invariant of orbit equivalence for actions of non discrete locally compact groups (see Remark \ref{rmq: OE and cost in lc setting}).


\begin{thmi}[see Theorem \ref{thm:toporank g delta}]\label{thmi:toporanklc}
Let $G$ be a locally compact unimodular non-discrete non-compact Polish group. For every measure-preserving, essentially free and ergodic action of $G$, there is a dense $G_\delta$ of couples $(T,U)$ in $[\mathcal R_G]^2$ which generate a dense free subgroup of $[\mathcal R_G]$ acting freely. In particular, the topological rank of $[\mathcal R_G]$ is $2$. 
\end{thmi}

A key tool in the proof of Theorem \ref{thmi:toporanklc} is a well-known result of Forrest, namely the existence of a \textit{cross-section} for actions of locally compact group \cite{MR0417388}. This will roughly provide a countable group $\Gamma$ such that the cost $1$ group $\Z\times \Gamma$ is a ‘‘measurable dense subgroup'' of $G$. We can then use the results of \cite{gentopergo} along with Theorem \ref{thmi:density} to find a dense $2$-generated subgroup.

\begin{rmqi}
Theorem \ref{thmi:toporanklc} is also true in the case $G$ is compact acting ergodically on $(X,\mu)$. Indeed in this case the action is essentially transitive and $[\mathcal R_G]=\Aut(X,\mu)$ which has a dense $G_\delta$ of couples of topological generators inducing a free action of the free group on two generators by results of Prasad and Törnquist respectively \cite{MR624915,MR2210067}. We do not know whether Theorem \ref{thmi:toporanklc}  holds for non-discrete Polish groups, even in the case of suitable actions. 
\end{rmqi}
\vspace{0.2cm}

In this work, we also extend a result of Giordano and Pestov, Theorem 5.7 of \cite{MR2311665}, that says that a full group of a free, ergodic action of a countable group is \textit{extremely amenable} if and only if the acting countable group is amenable. 

Before stating the theorem, let us recall that a group is \textit{extremely amenable} if every action of the group on a compact space admits a fixed point. The first example of an extremely amenable group was given by Christensen and Herer \cite{MR0412369}. Since then several examples of extremely amenable groups have been found such as the unitary group of a separable Hilbert space \cite{MR708367} or the group of measure-preserving bijections of a standard probability space \cite{MR1891002}.

\begin{thmi}\label{thmi:extr}
  Let $G$ be a locally compact second countable unimodular group acting freely and ergodically on $(X,\mu)$. Then the full group of $G$ is extremely amenable if and only if $G$ is amenable. 
\end{thmi}

The proof of the direct implication in the theorem is an easy adaptation of Giordano and Pestov's arguments to the locally compact case, using cross-sections.

For the other direction, we follow a different path and use von Neumann algebras. We first prove that the von Neumann algebra of the action $G\act(X,\mu)$ is generated by the full group $[\mathcal R_G]$, see Proposition \ref{prop:full group gen vna}. We use this to show that if $[\RR_G]$ is extremely amenable, then the von Neumann algebra of the action is amenable and therefore the acting group is amenable.


\subsection*{Acknowledgement} 

Both authors would like to thank Damien Gaboriau, Henrik Petersen, Sven Raum and Todor Tsankov for many useful discussions around this topic. We are especially grateful to Sven Raum for explaining to us the proof of item (3) of Proposition \ref{prop:full group gen vna} in the discrete case.  The authors were partially supported by Projet ANR-14-CE25-0004 GAMME.

\section{Orbit equivalence in the locally compact case}

\subsection{Measure-preserving actions}
Whenever a group $G$ acts on a set $X$ and $x\in X$, we denote by $G_x\leq G$ the stabilizer of $x$. 
The \textbf{free part} of an action $G\act X$ is the $G$-invariant set of all $x\in X$ such that $G_x=\{e\}$. 

A standard probability space is a probability space $(X,\mathcal B,\mu)$ such that $(X,\mathcal B)$ is a standard Borel space and $\mu$ is a Borel non-atomic probability measure. All such probability spaces are isomorphic, see \cite[Thm. 17.41]{MR1321597}. A subset $A$ of $X$ is a \textbf{Borel set} if it belongs to the $\sigma$-algebra $\mathcal B$. It is called a (Lebesgue-) \textbf{measurable set} if it belongs to the $\mu$-completion of $\mathcal B$. From now on, we will drop the $\mathcal B$ and fix  a standard probability space $(X,\mu)$. 

Whenever $G$ is a Polish group, a \textbf{Borel $G$-action} is a Borel action map $\alpha: G\times X\to X$. As usual, we will often drop the letter $\alpha$ and let $g\cdot x:=\alpha(g,x)$ for every $g\in G$ and $x\in X$.
 The following lemma is  well-known, for a proof see \cite[Lem. 10]{zbMATH06244595}.

\begin{lem}\label{lem:free part is Borel}Let $G$ be a locally compact Polish\footnote{Recall that a locally compact group is Polish if and only if it is second-countable (see \cite[Theorem 5.3]{MR1321597}).} group, and consider a Borel $G$-action on a standard Borel space $X$. Then the free part of the $G$-action is a Borel subset of $X$.
\end{lem}

We denote by $\Aut(X,\mu)$ the group of all measure-preserving Borel bijections of $(X,\mu)$, where we identify two such bijections if they coincide on a full measure subset of $X$. It is equipped with the \textbf{weak topology}, defined to be the coarser group topology which makes the maps $T\in\Aut(X,\mu)\mapsto \mu(T(A)\bigtriangleup A)$ continuous for every Borel set $A$. This turns $\Aut(X,\mu)$ into a Polish group (see e.g. \cite[I.1.(B)]{MR2583950}).

A \textbf{measure-preserving} $G$-action on $(X,\mu)$ is a Borel $G$-action on $X$ such that for every $g\in G$ and every Borel $A\subseteq X$, one has $\mu(g A)=\mu(A)$. If $G$ is a group, a near-$G$-action on $(X,\mu)$ is a homomorphism $G\to\Aut(X,\mu)$. Every measure-preserving action induces a near-action, and a near-action is the same as an action by $\mu$-preserving automorphisms on the measure algebra of $(X,\mu)$.

The following lemma is well-known when $G$ is locally compact (see for instance \cite[Lem. II.1.1]{MR910005}). We include a simple proof which works for all Polish groups. 

\begin{lem}Every measure-preserving action $\alpha$ of a Polish group $G$ on $(X,\mu)$ induces a continuous near-action $\rho_\alpha:G\to\Aut(X,\mu)$.
\end{lem}
\begin{proof}
  By Pettis' Lemma (see \cite[Thm. 1.2.6]{MR1425877}), we only need to check that $\rho_\alpha$ is a Borel map. By definition of the weak topology on $\Aut(X,\mu)$ it is enough to show that for every Borel subset $A$ of $X$ and every $\epsilon>0$, the set \[B:=\{g\in G:\mu(g(A)\bigtriangleup A)<\epsilon\}\] is Borel. For this, observe that since the action is Borel, the subset $\Gamma:=\{(g,x)\in X\times G:\ x\in g(A)\}$ is Borel and hence $\Gamma_A:=\Gamma\bigtriangleup (A\times G)$ is also Borel. This implies that the map \[M:g\mapsto \mu\left(\left\lbrace x\in X:\ x\in A\bigtriangleup g(A)\right\rbrace\right)\] is also Borel. So we can conclude observing that $B=M^{-1}([0,\epsilon[)$. \end{proof}

For locally compact Polish  groups, measure-preserving actions and  continuous near-actions are in one-to-one correspondence.
\begin{thm}[Mackey, \cite{MR0143874}]\label{thmmac}
  Let $G$ be a locally compact Polish group and let $(X,\mu)$ be a standard probability space. Then for every continuous homomorphism $\rho:G\to \Aut(X,\mu)$ there exists a measure-preserving action $\alpha$ of $G$ on $(X,\mu)$ such that the induced homomorphism $\rho_\alpha: G\to \Aut(X,\mu)$ is equal to $\rho$.

Moreover if $\alpha$ and $\beta$ are two measure-preserving actions of $G$ such that the induced homomorphisms $\rho_\alpha$ and $\rho_\beta$ are equal, then there is a Borel $G$-invariant subset $A\subset X$ of full measure such that $\alpha\bigr|_A=\beta\bigr|_A$. 
\end{thm}

\begin{rmq} The above result is in sharp contrast with the following situation which was uncovered by Glasner, Tsirelson and Weiss: if $G$ a Levy Polish group, every measure-preserving $G$-action is trivial but $G$ can still have interesting continuous near actions. Examples include $\Aut(X,\mu)$ itself or the orthogonal group of an infinite-dimensional Hilbert space, see \cite{MR2191233}.
\end{rmq}

If $G$ is locally compact and Polish, then we will call the measure-preserving action associated to a near-action $\rho:G\to \Aut(X,\mu)$ a realization of the near-action. Let us recall two important definitions. 

\begin{df}
  A measure-preserving action of a Polish group $G$ on the probability measure space $(X,\mu)$ is 
  \begin{itemize}
  \item \textbf{essentially free} if the free part of the action has full measure, that it if there is a measurable subset of full measure $A\subset X$ such that for every $x\in A$ and every $g\in G$, we have that $gx\neq x$; 
  \item \textbf{ergodic} if every Borel subset $A\subset X$ which is \textbf{almost $G$-invariant} (i.e. for all $g\in G$ we have $\mu(A\bigtriangleup g(A))=0$) has measure $0$ or $1$. 
  \end{itemize}
\end{df}

\begin{rmq}
  \begin{itemize}
  \item In the definition of essential freeness, one can actually asssume that $A$ is $G$-invariant and Borel by Lemma \ref{lem:free part is Borel}.
  \item There are actions of compact groups such that for every $g\in G$, the set $\{x\in X:\ gx=x\}$ has measure $0$ but which are not essentially free.
  \item Mackey's Theorem implies that if $G$ is locally compact and Polish then if a realization of an action is essentially free, then all Borel realizations are essentially free. 
  \item Another easy consequence of Mackey's Theorem \ref{thmmac} is that a measurable  action of a locally compact Polish group is ergodic if and only if  every Borel subset $A\subset X$ which is  $G$-invariant (i.e. for all $g\in G$, we have $g(A)=A$) has measure $0$ or $1$ (see \cite[Theorem 3]{MR0143874}). This is not true for Polish groups in general, as witnessed by Kolmogorov's example \cite[Example 9]{MR1760945}.
\end{itemize}
\end{rmq}

Every locally compact Polish group admits an essentially free measure-preserving action (see e.g. Proposition 1.2 in \cite{MR1250814}). We will now give  two concrete examples of measure-preserving actions.

\begin{exemple}
Suppose that $G$ is totally disconnected and non-compact. By van Dantzig's theorem, there exists a chain $(K_n)_{n\in\N}$ of compact open subgroups of $G$ such that $\cap_n K_n=\{1_G\}$. We let now $G$ act by permutations on the countable set $\sqcup_n G/K_n$. The associated  Bernoulli shift on $[0,1]^{\sqcup_n G/K_n}$ is essentially free since the $G$-action on $\sqcup_n G/K_n$ is faithful, and ergodic because every $G$-orbit on $\sqcup_n G/K_n$ is infinite.
\end{exemple}

\begin{exemple}\label{ex:lattices1}Suppose that $G$ has a lattice $\Gamma<G$, let $\lambda$ be a Haar measure on $G$ and let $D$ be a fundamental domain of the right $\Gamma$-action on $G$. Then any probability measure-preserving action of $\Gamma$ on $(X,\mu)$ induces a measure-preserving action of $G$ on $(X\times D,\mu\times \lambda|_D)$, see Definition 4.2.21 in \cite{MR776417}.
\end{exemple}

We will see in Section \ref{sec:cross-sec} that all actions of any locally compact Polish group can be decomposed as a product equivalence relation as in the previous example.

\subsection{Orbit full groups}\label{sectfull}

Let us start by recalling Dye's definition of full groups \cite{MR0131516}. A subgroup $\mathbb G\leq \Aut(X,\mu)$ is \textbf{full} if whenever $(A_n)$ is a partition of a full measure subset of $(X,\mu)$ and $(T_n)_{n\in\N}$ is a sequence of elements of $\mathbb G$, the new element $T\in\Aut(X,\mu)$ defined by 
$$T(x)=T_n(x)\text{ for all } x\in A_n$$
actually belongs to $\mathbb G$. A full group is \textbf{ergodic} if for all $A\subseteq X$, if for all $g\in\mathbb G$ one has $\mu(gA\bigtriangleup A)=0$, then $A$ has measure $0$ or $1$. Given a group $G\leq \Aut(X,\mu)$, there is a smallest full group containing $G$, denoted by $[G]$. If the corresponding $G$-almost action is ergodic, then $[G]$ is ergodic. The following proposition is well known in the case of full groups of ergodic measure-preserving equivalence relations. Its proof in the general case can be found in \cite[Lem. 3.2]{MR0131516}.  

\begin{prop}\label{prop: ergodic implies transitive full group}Let $\mathbb G\leq \Aut(X,\mu)$ be an ergodic full group. Then for any $A,B\subseteq X$ of same measure, there is $T\in \mathbb G$ such that $T(A)=B$ up to measure zero. 
\end{prop}

For a measure-preserving action of a Polish group $G$ on $X$, we will denote by $\mathcal R_G$ the \textit{orbit equivalence relation}, \[\mathcal R_G:=\left\lbrace (x,y)\in X\times X:\ \text{ there exists }g\in G\text{ such that }g\cdot x=y\right\rbrace.\]

We recall now the definition of orbit full groups and their Polish topology (see \cite{Carderi:2014qr} for more details and proofs). 
\begin{df}
  Let $G$ be a Polish group. The \textbf{orbit full group} of a probability measure-preserving action of $G$ on $(X,\mu)$ is the group \[[\mathcal R_G]:=\{T\in \Aut(X,\mu):\ T(x)\in G\cdot x\text{ for every }x\in X\}.\] 
\end{df}

Let $d_G$ be a compatible, right-invariant and bounded metric on $G$. We denote by $\LL^0(X,\mu,G)$ the space of measurable functions from $X$ to $G$ which we equip with the metric \[\tilde d_G(f,g):=\int_X d_G(f(x),g(x)) d\mu(x).\]

The topology induced by this metric only depends on the topology of $G$. It is a Polish topology called the topology of \textit{convergence in measure}. 

For a probability measure-preserving action of $G$ on $(X,\mu)$, for every measurable subset $A\subset X$ and measurable function $f:A\to G$, we define \[\Phi(f):A\to X,\text{ by }\Phi(f)(x)=f(x)x,\]
and we put $\widetilde{[\mathcal R_G]}:=\Phi^{-1}([\mathcal R_G])$.

The Polish space $\widetilde{[\mathcal R_G]}$ equipped with the product $f\cdot g(x)=f(g(x)x)g(x)$ becomes a Polish group for the topology of convergence in measure. Moreover the map $\Phi:\widetilde{[\mathcal R_G]}\to[\RR_G]$ is a group homomorphism with respect to this product.

The topology of \textbf{convergence in measure} on $[\mathcal R_G]$ is the quotient topology induced by $\Phi$ and we proved in Theorem 1 of \cite{Carderi:2014qr} that is a Polish group topology. 

\begin{rmq}
  If the action of $G$ is essentially free, then the map $\Phi$ is a bijection, so the convergence in measure on $[\mathcal R_G]$ is given by the metric $\tilde d_G$. 
\end{rmq}

We also recall that full groups of locally compact Polish groups are complete invariants of orbit equivalence. 

\begin{df}
  A probability measure-preserving action of the group $G$ on $(X,\mu)$ is \textbf{orbit equivalent} to a probability measure-preserving action of the group $H$ on $(Y,\nu)$ if there exists a subset of full measure $A\subset X$ and a measure-preserving Borel bijection $\ph:A\to Y$ such that \[\ph\times \ph(\mathcal R_G\cap (A\times A))=\mathcal R_H\cap (\ph(A)\times \ph(A)).\]
\end{df}

We will also say that the equivalence relations $R_G$ and $R_H$ are \textit{isomorphic up to measure zero}. We recall the following theorem.

\begin{thm}[{\cite[Thm. 3.26]{Carderi:2014qr}}]
  Let $G$ and $H$ be locally compact Polish groups acting on the probability space $(X,\mu)$ preserving the measure. Then the actions are orbit equivalent if and only if the associated orbit full groups are isomorphic.
\end{thm}


\subsection{Cross-sections and product decomposition}\label{sec:cross-sec}

We present now the most important property of measure-preserving actions of locally compact Polish groups: the existence of a \textit{cross-section}. 

\begin{df}\label{df:crosssection}
  Consider an essentially free, measure-preserving action of a locally compact Polish group $G$ on a standard probability space $(X,\mu)$. A Borel subset $Y\subset X$ is a \textbf{cross-section} of the action if there exists a neighborhood of the identity $U\subset G$ such that the map $\theta:U\times Y\to X$ defined by $\theta(u,y):=uy$ is injective and such that $\mu(X\setminus GY)=0$. 
\end{df}

The existence of cross section was proved by Forrest in \cite[Proposition 2.10]{MR0417388} in the more general context of \textit{non-singular} actions. For a more recent proof, we invite the reader to read Theorem 4.2 of \cite{Kyed:2013lq}. The following theorem is essentially a version of \cite[Proposition 2.13]{MR0417388} in the context of a measure-preserving action of a unimodular locally compact group. 

\begin{thm}\label{thm:OEprod}
Let $G$ be a unimodular, locally compact, non-compact and non-discrete Polish group, denote by $\lambda$ a Haar measure on $G$. Consider a measure-preserving, essentially free and ergodic action of $G$ on the standard probability space $(X,\mu)$. 
\begin{enumerate}[(i)]
\item There exists a standard probability space $(Y,\nu)$ and a countable group $\Gamma$ acting on $(Y,\nu)$ by measure-preserving transformations such that the action of $G$ is orbit equivalent to the product action of $\mathbb S^1\times \Gamma$ on $(\mathbb S^1\times Y,L\times \mu)$, where $\mathbb S^1$ is the circle group\footnote{Actually, any infinite compact metrizable group will do; the point is that the orbit equivalence relation associated to $\mathbb S^1\act\mathbb S^1$ is transitive.} acting on itself by translation and $L$ is its normalized Lebesgue measure. 

 \item Identifying $G\times X$ to $\mathcal R_G$ via the map $(g,x)\mapsto (g,g\cdot x)$, one can choose an orbit equivalence map $\Theta:(\circle\times Y,h\times \nu)\to (X,\mu)$ such that the induced map between equivalence relations \begin{align*}
 \Theta\times\Theta: &(\mathcal R_\circle\times\mathcal R_\Gamma,L\times L\times\tilde{\nu})\to (\mathcal R_G,\Lambda\times\mu)\end{align*} is measure-preserving, where $\tilde{\nu}$ is the $\sigma$-finite measure induced by $\nu$ on $\RR_\Gamma$ via integration of the counting measure of the fibers. 
 
 \item The group $G$ is amenable if and only if  the orbit equivalence relation induced by $\Gamma$ on $(Y,\nu)$ is amenable.
\end{enumerate}
\end{thm}
\begin{proof} 
  In this proof we will use the notations and conventions of Proposition 4.3 of \cite{Kyed:2013lq}. Let $Y\subset X$ be a cross section and let $U\subset G$ be a neighborhood of the identity as in Definition \ref{df:crosssection}. We consider the restriction of $\RR_G$ to $Y$, \[\mathcal R:=\{(y,y')\in Y\times Y : \exists g\in G, y'=gy\}.\]

  By \cite[Proposition 4.3.1]{Kyed:2013lq}, $\mathcal R$ is a Borel, countable equivalence relation on $Y$. Define $\Psi:G\times X\to \RR_G$ by $\Psi(g,x)=(gx,x)$ and observe that since the action is free, $\Psi$ is a bijection. Denote by $\Psi_Y$ the restriction of $\Psi$ to $G\times Y$ and put $Z=\Psi_Y(G\times Y)$. Observe that the projection $\pi:Z\to X$ on the first coordinate is countable-to-one hence we can define a measure $\eta$ on $Z$ by integrating with respect to $\mu$ the counting measure over the projection $\pi$. By definition we have $\mu(UY)=\eta(\Psi_Y(U\times Y))$. Put $\covol(Y):=\lambda(U)/\mu(UY)$. As explained in the proof of \cite[Proposition 4.3.2]{Kyed:2013lq}, the unicity of the Haar measure on $G$ implies the existence a probability measure $\nu$ on $Y$ such that $\Psi_*(\lambda\times (\nu/\covol(Y)))=\eta$. 

Moreover by \cite[Proposition 4.3]{Kyed:2013lq}, we know that

 \begin{enumerate}[(1)]
 \item the probability measure $\nu$ is $\mathcal R$-invariant,
 \item $(\mathcal R,\nu)$ is ergodic if and only if the action of $G$ is ergodic,
 \item $(\mathcal R,\nu)$ has infinite orbits almost everywhere if and only if $G$ is non-compact,
 \item $(\RR, \nu)$ is amenable if and only if $G$ is.
 \end{enumerate}

By property $(3)$ above, we deduce that $(Y,\nu)$ is diffuse. Moreover since $\mathcal R$ is countable and measure-preserving, Feldman and Moore's result (\cite[Theorem 1]{MR0578656}) gives us a measure-preserving action of a countable group $\Gamma$ on $(Y,\nu)$ which induces the equivalence relation $\mathcal R$. 

Up to taking an open subset of $U$, we may assume that $\mu( U\cdot Y)=\frac 1 K$ for some integer $K\in\N$. Set $A=U\cdot Y$. By ergodicity of $G$, we can find $T\in[\mathcal R_G]$ of order $K$ such that $\{A,T(A),...,T^{K-1}(A)\}$ is a partition of a full measure subset of $X$. Let us denote by $c$ the counting measure on $\Z/K\Z$ and consider the equivalence relation $\mathcal S'$ on $(\Z/K\Z\times U\times Y,c\times\lambda_U\times \nu)$ defined by \[(k,u,y)\mathcal S (k',u',y')\text{ if }y\mathcal R y'.\]

The measure-preserving map 
\begin{align*}
&\Theta': (\Z/K\Z\times U\times Y,c\times\lambda_U\times \nu)\to (X,\mu)\\
&\Theta'(k,u,y):=T^k(u\cdot y),
\end{align*}
defines an orbit equivalence between $\mathcal S'$ and $\mathcal R_G$. Denote by $L$ the normalized Lebesgue measure on the circle $\mathbb S^1$ and fix a measure-preserving isomorphism \[\alpha:(\mathbb S^1,L) \to (\Z/K\Z\times U,c\times \lambda_U).\]

Let $\mathcal S$ be the equivalence relation induced by the action of $\Gamma\times\circle $ on $Y\times \circle$ where $\Gamma$ acts on $\Gamma$ and $\circle$ acts on itself by translation. Observe that $\alpha$ induces an orbit equivalence between $\mathcal S'$ and $\mathcal S$, which combined with $\Theta'$ gives an orbit equivalence $\Theta$ between $\mathcal S$ and $\RR_G$ and hence $(i)$ is proved.

Now $(ii)$ can be deduced by an easy computation; it is also a direct application of the uniqueness of the Haar measure on $\mathcal R_G$  (see Theorem \ref{thm:haaruni}). Condition $(iii)$ follows from property $(4)$.
%
%
%
\end{proof}

\subsection{Weak orbit equivalence versus orbit equivalence}

Let  $\mathcal R$ be a Borel equivalent relation on $X$ and $A\subseteq X$ be a Borel subset with positive measure. The\textbf{restriction} of $\mathcal R$ to $A$ is the equivalence relation $\mathcal R\cap (A\times A)$ on the standard probability space $(A,\mu_A)$ where the measure $\mu_A$ is defined by: for all Borel $B\subseteq A$, $\mu_A(B)=\frac{\mu(B)}{\mu(A)}$.
Let us recall two important definitions. 
\begin{df} Let $\Gamma$ and $\Lambda$ be two countable groups.
  \begin{itemize}
  \item   The groups $\Gamma$ and $\Lambda$ are \textbf{measure equivalent} if there exists a standard $\sigma$-finite measured space $(\Omega,m)$ and commuting measure-preserving actions of $\Gamma$ and $\Lambda$ on $(\Omega,m)$ which are essentially free and admit a fundamental domain with finite measure.
  \item   The groups $\Gamma$ and $\Lambda$ are \textbf{weakly orbit equivalent} if $\Gamma$ and $\Lambda$ admit measure preserving essentially free ergodic actions on probability spaces $(X,\mu)$ and $(Y,\nu)$ such that there exist measurable subsets $A\subset X$ and $B\subset Y$ such that $\RR_\Gamma$ restricted to $A$ is orbit equivalent to $\RR_\Lambda$ restricted to $B$. The quantity $\mu(A)/\nu(B)$ is called the \textit{coupling constant} of the stable orbit equivalence. 
  \end{itemize}
\end{df}

\begin{rmq}An application of the ergodic decomposition theorem yields than one can drop the ergodicity assumption in the definition of weak orbit equivalence. 
\end{rmq}

Furman proved in \cite{MR1740985} that two countable groups are measure equivalent if and only if 
they are weakly orbit equivalent. Let us now study these notions for non-discrete locally compact groups.

\begin{df}Let $G$ and $H$ be two Polish locally compact groups.
\begin{itemize}
 \item The groups $G$ and $H$ are \textbf{orbit equivalent} if they admit ergodic measure-preserving essentially free actions which are orbit equivalent. 
 \item The groups $G$ and $H$ are \textbf{weakly orbit equivalent} if they admit ergodic measure-preserving essentially free actions on $(X,\mu)$ such that there exist positive measurable subsets $A\subset X$ and $B\subset X$ such that $\RR_G$ restricted to $A$ is orbit equivalent to $\RR_G$ restricted to $B$.
 \end{itemize}
\end{df}

Note that a non-discrete locally compact group is never weakly orbit equivalent to a discrete one.

\begin{lem}Let $G$ be a non-discrete locally compact Polish group acting essentially freely ergodically in a measure-preserving manner on $(X,\mu)$. Then for every Borel subset $A\subseteq X$ of positive measure, $\mathcal R_G$ is orbit equivalent to its restriction to $A$.
\end{lem}
\begin{proof}
By Theorem \ref{thm:OEprod}, we can find a standard probability space $(Y,\nu)$ and a countable group $\Gamma$ acting on $(Y,\nu)$ by measure-preserving transformations such that the action $\mathbb S^1\times \Gamma$ on $(\mathbb S^1\times Y,L\times \mu)$ is orbit equivalent to $\mathcal R_G$. It thus suffices to show that for every measurable $A\subseteq \circle\times Y$ of positive measure, $\mathcal R_{\circle\times\Gamma}= \mathcal R_{\circle}\times\mathcal R_\Gamma$ is orbit equivalent to its restriction to $(A,\mu_A)$.

So let $A$ be a subset of $\mathbb S^1\times Y$ of positive measure. Let $\tilde A$ be a Borel subset of $\mathbb S^1$ with measure $\mu(A)$. Then since $\mathcal R_{\circle}$ is transitive it is orbit equivalent to its restriction to $\tilde A$, which in turn yields that $\mathcal R_\Gamma\times \mathcal R_{\circle}$ is orbit equivalent to its restriction to $\tilde A\times Y$.

Since the $G$-action is ergodic, there exists $\varphi\in[\mathcal R_{\circle\times \Gamma}]$ which maps a full measure subset of $\tilde A\times Y$ to a full measure subset of $A$ (see Prop. \ref{prop: ergodic implies transitive full group}), so that the restrictions of $\mathcal R_{\circle\times \Gamma}$ to $A$ and $\tilde A\times Y$ are orbit equivalent. We conclude that $\mathcal R_{\circle\times \Gamma}$ is orbit equivalent to its restriction to $A$, hence the same conclusion holds for $\mathcal R_G$.
\end{proof}

\begin{thm}\label{thm: OE equiv SOE}Let $G$ and $H$ be two non-discrete locally compact Polish groups. Then $G$ and $H$ are orbit equivalent if and only if they are weakly orbit equivalent.
\end{thm}
\begin{proof}
The direct implication is by definition and the converse is a straightforward application of the previous lemma.
\end{proof}

\begin{rmq}\label{rmq: OE and cost in lc setting}
Any reasonable definition of cost for non-discrete locally compact groups would thus only provide three distinct orbit equivalence classes: the ones with a corresponding $\mathcal R_\Gamma$ of cost 1, the ones with $\mathcal R_\Gamma$ of finite cost greater than $1$, and the ones with $\RR_\Gamma$ of infinite cost. This mirrors the first $\LL^2$ Betti number of locally compact unimodular groups, see \cite{MR3072156} and \cite{Kyed:2013lq}.
\end{rmq}

The following proposition is surely well-know to experts, but we were not able to find it in the literature. 
It guarantees that orbit equivalence for locally compact non discrete groups is at least as complicated as measure equivalence for countable groups. 

\begin{prop}\label{prop:meoe}
  Two countable groups $\Gamma$ and $\Lambda$ are measure equivalent if and only if the locally compact groups $\Gamma\times \circle$ and $\Lambda\times \circle$ are orbit equivalent.
\end{prop}
\begin{proof}
  Suppose that $\Gamma\times \circle$ admits an action on the probability space $(X,\mu)$ which is orbit equivalent to an action of $\Lambda\times\circle$. Then we can let $\Gamma$ act on $\RR_{\Gamma\times\circle}$ on the left and we can let $\Lambda$ act on the right via the orbit equivalence. These two actions commutes and in both cases a fundamental domain is given by $X\times \circle$ which has finite measure.

  Suppose now that the groups $\Gamma$ and $\Lambda$ are measure equivalent, then by \cite[Lem. 2.2.2]{MR1740986} and \cite[Lem. 3.2]{MR1740985}  they are weakly orbit equivalent. Then $\Gamma\times \mathbb S^1$ and $\Lambda\times \mathbb S^1$ are weakly orbit equivalent, hence orbit equivalent by Theorem \ref{thm: OE equiv SOE}.
  \end{proof}

Note that the groups $\Gamma\times\circle$ and $\Lambda\times\circle$ are unimodular, so Theorem \ref{thm:haaruni} applies. Therefore any orbit equivalence between them sends the Lebesgue measure of the circle group to a multiple of the Lebesgue measure of the other circle. This constant is the coupling constant of the induced measure equivalence between $\Gamma$ and $\Lambda$.

\begin{cor}There are uncountably many non-discrete locally compact groups up to orbit equivalence.
\end{cor}
\begin{proof}This follows from the fact that there are uncountably many countable groups up to measure equivalence (see the paragraph preceding $\mathbf P_{\mathrm{ME}}\mathbf{15^*}$ in \cite{MR2183295}).
\end{proof}

\section{Dense subgroups in orbit full groups}

The aim of this section is to prove the following theorem.

\begin{thm}\label{thm:densesub}
  Let $G$ be a locally compact Polish group. For every measure-preserving action of $G$ on the probability space $(X,\mu)$ and for every dense subgroup $H\subset G$, we have that  $[\mathcal R_H]$ is dense in $[\mathcal R_G]$. 
\end{thm}

Theorem \ref{thm:densesub} will be a crucial tool to compute the topological rank\footnote{Recall that the topological rank of a topological group is the minimum of the rank of a countable dense subgroup.} of an orbit full group (see next section).

\subsection{Suitable actions}

We will prove Theorem \ref{thm:densesub} under a weaker hypothesis in the context of Polish group actions. Recall however, that Kolmogorov's example (\cite[Example 9]{MR1760945}) provide a counterexample of Theorem \ref{thm:densesub}. Indeed there is a Borel probability measure on $\{0,1\}^\N$ such that the full group generated by the finitely supported permutations is not dense in the orbit full group of the Polish group of all permutations of $\N$ acting by shift on $\{0,1\}^\N$. The fact is that this action is not \textit{suitable}.

 \begin{df}[Becker, {\cite[Definition 1.2.7]{MR2995370}}]\label{df:suitable}
Let $G$ be a Polish group. A Borel, measure-preserving action of $G$ on the probability space $(X,\mu)$ is \textbf{suitable} if for all Borel subsets $A,B\subset X$ of positive measure, one of the following two conditions holds:
\begin{enumerate}[(1)]
\item for any open neighborhood of the identity $O\subset G$, there is $g\in O$ such that $\mu(A\cap g B)>0$;
\item there are Borel subsets $A'\subset A$ and $B'\subset B$ of full measure in $A$ and $B$ and an open neighborhood $O$ of the identity in $G$ such that $(OA')\cap B'=\emptyset$. 
\end{enumerate}
\end{df}

We will prove the following. 

\begin{thm}\label{thm:densesubpol}
  Let $G$ be a Polish group. For every Borel, measure-preserving, suitable action of $G$ on the probability space $(X,\mu)$ and for every dense subgroup $H\subset G$, the orbit full group $[\mathcal R_H]$ is dense in $[\mathcal R_G]$. 
\end{thm}

Becker proved in Theorem 1.2.9 of \cite{MR2995370}, that all measure-preserving actions of locally compact Polish groups are suitable, so Theorem \ref{thm:densesubpol} implies Theorem \ref{thm:densesub}. We end this section by giving a different example to which our results apply.

\begin{exemple}\label{example: suitable bernoulli}The standard Bernoulli shift $\mathfrak S_\infty\act([0,1]^\N,\lambda^{\otimes\N})$ is a suitable action, where $\lambda$ is the Lebesgue measure. 
\end{exemple}
\begin{proof} Set $X:=[0,1]^\N$ and $\mu:=\lambda^{\otimes\N}$ and let $A,B\subseteq X$ be Borel subsets of positive measure. Suppose that (1) does not hold, and let $O$ be a neighborhood of the identity such that
\[\text{for all }g\in O,\ \mu(A\cap g B)=0.\]
By shrinking $O$ if necessary, we may assume that there exists $N\in\N$ such that $O$ is the subgroup of $\mathfrak S_\infty$ consisting of all the permutations which fix pointwise the set $\{0,...,N-1\}$. We endow $Y:=[0,1]^{\{0,...,N-1\}}$ with the measure $\eta:=\lambda^{\otimes N}$ and $Z:=[0,1]^{\N\setminus\{0,...,N-1\}}$ with the measure $\nu:=\lambda^{\N\setminus\{0,...,N-1\}}$. Observe that $(X,\mu)=(Y,\eta)\times(Z,\nu)$.

Let $\tau\in O$ be a bijection whose set of fixed points is  $\{0,...,N-1\}$, and which has only one non-trivial orbit. Then its Bernoulli shift on $(Z,\nu)$ is ergodic since it is conjugate to the $\Z$-shift on $([0,1]^\Z,\lambda^{\otimes\Z})$. Moreover the ergodic decomposition of its Bernoulli shift on $(X,\mu)$ is given by $(\mu_y)_{y\in Y}$ where $\mu_y$ is the probability measure on $X$ defined by $\mu_y:=\delta_y\otimes \nu$. 

Since $\mu(A\cap \tau^k B)=0$ for all $k\in\Z$ and every $\mu_y$ is $\tau$-ergodic, there is a full measure Borel subset $Y'$ of $Y$ such that for all $y\in Y'$ we have that $\mu_y(A)=0$ whenever $\mu_y(B)>0$. Let $Y_0$ be the Borel set of $y\in Y'$ such that $\mu_y(B)>0$, and put $Y_1:=Y'\setminus Y_0$. The sets $A':=(Y_1\times [0,1]^\N)\cap A$ and $B':=(Y_0\times [0,1]^\N)\cap B$ witness that (2) holds. 
\end{proof}

\begin{rmq}
In the above example, since the countable group $\mathfrak S_{(\infty)}$ of finitely supported permutations is amenable, the full group it generates is extremely amenable for the uniform topology by \cite[Thm. 5.7]{MR2311665}. Since the uniform topology refines the topology of convergence in measure and $\mathfrak S_{(\infty)}$ is dense in $\mathfrak S_\infty$, Theorem \ref{thm:densesubpol} yields that $[\mathcal R_{\mathfrak S_\infty}]$ is extremely amenable. It would be interesting to understand for which non locally compact closed subgroups $G\leq \mathfrak S_\infty$ the orbit full group associated to the standard Bernoulli shift is extremely amenable. 
\end{rmq}

\subsection{An equivalent statement}

From now on, we will use the notations of Section \ref{sectfull}. For every Borel, measure-preserving action of $G$ on the probability space $(X,\mu)$, we denote by $\widetilde{[G]}_D\subset \widetilde{[\mathcal R_G]}$ the subset of function with countable (essential) image and we put $[G]_D:=\Phi(\widetilde{[G]}_D)$. Note that $[G]_D$ is the smallest full group containing the image of $G$ inside $\Aut(X,\mu)$. 

Theorem \ref{thm:densesubpol} follows form the following weaker theorem, which is also important in its own right. 

\begin{thm}\label{thm:fmlsui}
  Let $G$ be a Polish group. For every Borel, measure-preserving, suitable action of $G$ on the probability space $(X,\mu)$, we have that $\widetilde{[G]}_D\subset \widetilde{[\mathcal R_G]}$ is a dense subgroup. 
\end{thm}

Before deducing Theorem \ref{thm:densesubpol} from the above result, we need the following lemma which will be used several times. 

\begin{lem}\label{lem: ext countable range} Let $A\subseteq X$ and let $f:A\to G$ be a function with countable image. If the map $\Phi(f):A\to X$ defined by $\Phi(f)(x)=f(x)x$ is injective, then there exists $f'\in\widetilde{[G]}_D$ which extends $f$.
\end{lem}
\begin{proof}
Let $\Gamma$ be the group generated by the range of the function $f$.
Since $\Phi(f)$ is injective,  the map $\Phi(f)$ is an element of the pseudo-full group\footnote{The pseudo-full group of the countable group $\Gamma$ acting on $(X,\mu)$ is defined to be the set of Borel partial maps whose graph is a subset of $\mathcal R_\Gamma$.} of $\Gamma$. The elements of the pseudo-full group of $\Gamma$ preserve the $\mathcal R_\Gamma$-conditional measure and any two sets having the same $\mathcal R_\Gamma$-conditional measure can be sent to one another by an elements of the pseudo full group of $\Gamma$(see \cite[Sec. 2.1]{lm14nonerg} for details). Therefore there is an element $T\in[G]_D$ which extends $\Phi(f)$. By lifting such a $T$ to $\widetilde{[G]}_D$ where $f$ was not defined, we obtain $f'\in\widetilde{[G]}_D$ which extends $f$.
\end{proof}

\begin{proof}[Proof of Theorem \ref{thm:densesubpol}]
  Let $G$ be a Polish group and let $H$ be a dense subgroup. Consider a Borel, measure-preserving suitable action of $G$ on the probability space $(X,\mu)$. By Theorem \ref{thm:fmlsui}, we only need to prove that $\widetilde{[H]}_D\subset \widetilde{[G]}_D$ is dense. 

Fix a compatible, right-invariant metric $d_G$ on $G$ bounded by $1$, fix $\eps>0$ and take $f\in \widetilde{[\mathcal R_G]}_D$. There are $k\in\N$, a finite subset $\{g_1,\ldots,g_k\}\subset G$ and a finite partition $\{A_0,\ldots,A_k\}$ of $X$ such that $\mu(A_0)\leq \eps/2$ and for every $i\geq 1$, we have $f(A_i)=\{g_i\}$. By density and weak-continuity of the action, there exists $\{h_1,\ldots,h_k\}\subset H$ such that for every $i\in \{1,\ldots,k\}$, we have that $d_G(g_i,h_i)\leq \eps$ and $\mu(g_i(A_i)\Delta h_i(A_i))\leq \eps/2k$. Put \[B_i:=h_i^{-1}( g_i(A_i)\cap h_i(A_i))\subset A_i,\quad B:=\cup_{i=1}^k B_i\] and observe that $\mu(B)\geq 1-\eps$. 
Consider the map  $f':B\to H$ defined by $f'(x)=h_i$ whenever $x\in B_i$.
Then $f$ has finite range, and since the subsets $h_iB_i$ are disjoint, $\Phi(f)$ is injective. Lemma \ref{lem: ext countable range} allows us to extend $f'$ to $f''\in \widetilde{[G]}_D$. 
We clearly have $\tilde d_G(f,f'')\leq 2\eps$, which ends the proof. 
\end{proof}

\subsection{Proof of Theorem \ref{thm:fmlsui}}

\begin{df}\label{df:good}
  Fix $f\in \widetilde{[\mathcal R_G]}$ and a neighborhood of the identity $N\subset G$. We say that a couple $(A,g)$ is ($N$-)\textbf{good} if 
\begin{enumerate}
\item $A\subset X$ is a measurable subset of positive measure and $g:A\to G$ is a measurable function with countable image,
\item for every $x\in A$, we have $f(x)g(x)\inv\in N$,
\item the map $\Phi(g):A\to X$ defined by $\Phi(g)(x)=g(x)x$ is injective. 
\end{enumerate}
\end{df}

We note that for a fixed $f\in\widetilde{[\mathcal R_G]}$ the existence of a good couple is not a trivial fact. Indeed, we will use the hypothesis that the action is suitable only to show the existence of such couples.

The proof of the theorem will be a measurable version of the Hall's marriage theorem and it will follow the same strategy as Hudson's in \cite{MR1207735}. For a fixed $f$ as in Definition \ref{df:good}, using Zorn's lemma, we will construct for every $\eps>0$ and neighborhood of the identity $N\subset G$ a good couple such that $\mu(A)>1-\eps$ in three steps.

\subsubsection*{Step 1}

\noindent In the first step (and only in this one), we will use the hypothesis that the action is suitable. 
\begin{prop}\label{prop:ess}
  Let $f\in\widetilde{[\mathcal R_G]}$ and let $N\subset G$ be a neighborhood of the identity. For every $B\subset X$ of positive measure, there is a good couple $(A,g)$ such that $A\subset B$ has positive measure and $\Phi(g)(A)\subset \Phi(f)(B)$. 
\end{prop}
\begin{proof}
Consider a neighborhood of the identity $O\subset G$ such that $O=O^{-1}$ and $O^2\subset N$. Let $f(x_0)$ be an element of the support of the pushforward measure $f_*\mu\bigr|_B$ and put $A_O:=B\cap f^{-1}(Of(x_0))$.  For every neighborhood of the identity $O'$ in $G$, set $C_{O'}:=B\cap f\inv(O'f(x_0))$. Note that $C_{O'}\subset A_O$, whenever $O'\subset O$. By definition of the support of $f_*\mu\bigr|_B$ the Borel set $C_{O'}$ has positive measure. 

 Let us show that condition $(2)$ of Definition \ref{df:suitable} is not satisfied for the two Borel sets $\Phi(f)(A_0)$ and $f(x_0)A_0$. Indeed, $\Phi(f)\inv$ and $f(x_0)\inv$ are measure-preserving so if condition $(2)$ holds, then there is a full measure subset $A'\subseteq A_O$ such that $\Phi(f)(A')$ and $O'f(x_0)A'$ are disjoint. This is a contradiction because $\Phi(f)(A')$ and $O'f(x_0)A'$ contain $\Phi(f)(A'\cap C_{O'})$ which has positive measure. 

Since the action is suitable,  $(1)$ (of Definition \ref{df:suitable}) has to hold. So there is $h\in O$ such that \[\mu(\Phi(f)(A_O)\cap h f(x_0)A_O)>0.\]
 Set $ A:=A_O\cap f(x_0)^{-1}h^{-1}\Phi(f)(A_O)$ and for every $x\in A$ put $g(x):=hf(x_0)$. The couple $(A,g)$ is good, because for every $x\in A$ we have that $f(x)f(x_0)^{-1}\in O$ and \[f(x)g(x)^{-1}=f(x)f(x_0)^{-1}h\in O^2\subset N.\qedhere\] 
\end{proof}

\subsubsection*{Step 2}

\noindent For a neighborhood $N$ of the identity in $G$ and $\eps>0$, we now define the order on the family of $N$-good couples associated to a function $f\in \widetilde{[\mathcal R_G]}$. 

\begin{df}
  Let $(A_1,g_1)$ and $(A_2,g_2)$ be two good couples. We say that $(A_1,g_1)\prec (A_2,g_2)$ if $A_2\supseteq A_1$ almost everywhere and if \[\mu(\lbrace x\in A_1:\ g_1(x)\neq g_2(x)\rbrace)\leq \frac{1}{\eps}(\mu(A_2)-\mu(A_1)).\]
\end{df}

We identify two good couples $(A_1,g_1)$ and $(A_2,g_2)$ if $\mu(A_1\bigtriangleup A_2)=0$ and for almost all $x\in A_1\cap A_2$, $g_1(x)=g_2(x)$. This makes the relation $\prec$ antisymmetric. It is moreover clearly reflexive.

\begin{lem}
The relation $\prec$ is an order relation on the set of good couples. 
\end{lem}
\begin{proof}
The only fact left to prove is that $\prec$ is transitive. For this suppose that \[(A_1,g_1)\prec (A_2,g_2)\prec (A_3,g_3),\] then \[\lbrace x\in A_1:\ g_1(x)\neq g_3(x)\rbrace\subset \lbrace x\in A_1:\ g_1(x)\neq g_2(x)\rbrace\cup \lbrace x\in A_2:\ g_2(x)\neq g_3(x)\rbrace,\]
so we get 
\begin{align*}
  \mu(\lbrace x\in A_1:\ &g_1(x)\neq g_3(x)\rbrace)
  \\\leq &\mu(\lbrace x\in A_1:\ g_1(x)\neq g_2(x)\rbrace)+\mu(\lbrace x\in A_2:\ g_2(x)\neq g_3(x)\rbrace) \\
  \leq &\frac{1}{\eps}(\mu(A_2)-\mu(A_1))+\frac{1}{\eps}(\mu(A_3)-\mu(A_2))\\
  = &\frac{1}{\eps}(\mu(A_3)-\mu(A_1)).\qedhere
\end{align*}
\end{proof}

The following proposition is the core of the proof of Theorem \ref{thm:fmlsui}.

\begin{prop}\label{prop:bitbig}
For every good couple $(A,g)$ with $\mu(A)< 1-\eps$, there exists a good couple $(A',g')$ such that $(A,g)\prec (A',g')$ and $\mu(A'\setminus A)>0$.
\end{prop}

We would like to say that for every good couple $(A,g)$ there is $B\subset X\setminus A$ such that $\Phi(f)(B)\cap \Phi(g)(A)=\emptyset$. When this is the case, we can conclude using Proposition \ref{prop:ess}. The problem is that this is not always possible, but it is possible in a finite number of steps.

\begin{figure}[h]
\centering

\tikzset{middlearrow/.style={
        decoration={markings,
            mark= at position 0.5 with {\arrow[scale=1.5]{#1}} ,
        },
        postaction={decorate}
    }
}

\definecolor{ffffff}{rgb}{1.,1.,1.}
\begin{tikzpicture}[line cap=round,line join=round,>=triangle 45,x=4cm,y=2cm]
\clip(0.5,0.6) rectangle (5.4,3.5);
\draw [line width=1.2pt] (1.,3.)-- (4.,3.);
\draw [line width=1.2pt] (4.,1.)-- (1.,1.);
\draw [line width=1.2pt] (1.2,3.1)-- (1.6,3.1);
\draw [line width=1.2pt] (2.2,3.1)-- (2.6,3.1);
\draw [line width=1.2pt] (3.2,3.1)-- (3.6,3.1);
\draw [line width=1.2pt] (1.2,0.9)-- (1.6,0.9);
\draw [line width=1.2pt] (2.2,0.9)-- (2.6,0.9);
\draw [line width=1.2pt] (3.2,0.9)-- (3.6,0.9);
\draw [line width=2.pt] (2.,0.95)-- (4.,0.95);
\draw [line width=2.pt] (1.,3.05)-- (3.,3.05);
\draw [ dash pattern=on 1pt off 1pt] (1.2,1.)-> (1.2,3.);
\draw (1.2,2.02142258262) -- (1.22754332051,2.);
\draw (1.2,2.02142258262) -- (1.17245667949,2.);
\draw (1.2,1.97857741738) -- (1.22754332051,1.95715483476);
\draw (1.2,1.97857741738) -- (1.17245667949,1.95715483476);
\draw (1.2,2.06426774786) -- (1.22754332051,2.04284516524);
\draw (1.2,2.06426774786) -- (1.17245667949,2.04284516524);
\draw [dash pattern=on 1pt off 1pt] (1.4,1.)-- (1.4,3.);
\draw (1.4,2.02142258262) -- (1.42754332051,2.);
\draw (1.4,2.02142258262) -- (1.37245667949,2.);
\draw (1.4,1.97857741738) -- (1.42754332051,1.95715483476);
\draw (1.4,1.97857741738) -- (1.37245667949,1.95715483476);
\draw  (1.4,2.06426774786) -- (1.42754332051,2.04284516524);
\draw (1.4,2.06426774786) -- (1.37245667949,2.04284516524);
\draw [dash pattern=on 1pt off 1pt] (1.6,1.)-- (1.6,3.);
\draw (1.6,2.02142258262) -- (1.62754332051,2.);
\draw (1.6,2.02142258262) -- (1.57245667949,2.);
\draw (1.6,1.97857741738) -- (1.62754332051,1.95715483476);
\draw (1.6,1.97857741738) -- (1.57245667949,1.95715483476);
\draw (1.6,2.06426774786) -- (1.62754332051,2.04284516524);
\draw (1.6,2.06426774786) -- (1.57245667949,2.04284516524);
\draw [dash pattern=on 1pt off 1pt] (2.2,1.)-- (2.2,3.);
\draw (2.2,2.02142258262) -- (2.22754332051,2.);
\draw (2.2,2.02142258262) -- (2.17245667949,2.);
\draw (2.2,1.97857741738) -- (2.22754332051,1.95715483476);
\draw (2.2,1.97857741738) -- (2.17245667949,1.95715483476);
\draw (2.2,2.06426774786) -- (2.22754332051,2.04284516524);
\draw (2.2,2.06426774786) -- (2.17245667949,2.04284516524);
\draw [dash pattern=on 1pt off 1pt] (2.4,1.)-- (2.4,3.);
\draw (2.4,2.02142258262) -- (2.42754332051,2.);
\draw (2.4,2.02142258262) -- (2.37245667949,2.);
\draw (2.4,1.97857741738) -- (2.42754332051,1.95715483476);
\draw (2.4,1.97857741738) -- (2.37245667949,1.95715483476);
\draw (2.4,2.06426774786) -- (2.42754332051,2.04284516524);
\draw (2.4,2.06426774786) -- (2.37245667949,2.04284516524);
\draw [dash pattern=on 1pt off 1pt] (2.6,1.)-- (2.6,3.);
\draw (2.6,2.02142258262) -- (2.62754332051,2.);
\draw (2.6,2.02142258262) -- (2.57245667949,2.);
\draw (2.6,1.97857741738) -- (2.62754332051,1.95715483476);
\draw (2.6,1.97857741738) -- (2.57245667949,1.95715483476);
\draw (2.6,2.06426774786) -- (2.62754332051,2.04284516524);
\draw (2.6,2.06426774786) -- (2.57245667949,2.04284516524);
\draw [dash pattern=on 1pt off 1pt] (3.2,1.)-- (3.2,3.);
\draw (3.2,2.02142258262) -- (3.22754332051,2.);
\draw (3.2,2.02142258262) -- (3.17245667949,2.);
\draw (3.2,1.97857741738) -- (3.22754332051,1.95715483476);
\draw (3.2,1.97857741738) -- (3.17245667949,1.95715483476);
\draw (3.2,2.06426774786) -- (3.22754332051,2.04284516524);
\draw (3.2,2.06426774786) -- (3.17245667949,2.04284516524);
\draw [dash pattern=on 1pt off 1pt] (3.4,1.)-- (3.4,3.);
\draw (3.4,2.02142258262) -- (3.42754332051,2.);
\draw (3.4,2.02142258262) -- (3.37245667949,2.);
\draw (3.4,1.97857741738) -- (3.42754332051,1.95715483476);
\draw (3.4,1.97857741738) -- (3.37245667949,1.95715483476);
\draw (3.4,2.06426774786) -- (3.42754332051,2.04284516524);
\draw (3.4,2.06426774786) -- (3.37245667949,2.04284516524);
\draw [dash pattern=on 1pt off 1pt] (3.6,1.)-- (3.6,3.);
\draw (3.6,2.02142258262) -- (3.62754332051,2.);
\draw (3.6,2.02142258262) -- (3.57245667949,2.);
\draw (3.6,1.97857741738) -- (3.62754332051,1.95715483476);
\draw (3.6,1.97857741738) -- (3.57245667949,1.95715483476);
\draw (3.6,2.06426774786) -- (3.62754332051,2.04284516524);
\draw (3.6,2.06426774786) -- (3.57245667949,2.04284516524);
\draw [middlearrow={stealth}, dotted] (2.2,1.)-- (1.2,3.);
\draw [middlearrow={stealth}, dotted] (2.4,1.)-- (1.4,3.);
\draw [middlearrow={stealth}, dotted] (2.6,1.)-- (1.6,3.);
\draw [middlearrow={stealth}, dotted] (3.2,1.)-- (2.2,3.);
\draw [middlearrow={stealth}, dotted] (3.4,1.)-- (2.4,3.);
\draw [middlearrow={stealth}, dotted] (3.6,1.)-- (2.6,3.);
\draw [line width=2.pt,color=ffffff] (1.4,1.6)-- (1.4,1.4);
\draw [line width=2.pt,color=ffffff] (2.,1.8)-- (2.10324818062,1.59350363876);
\draw [line width=2.pt,color=ffffff] (2.4,1.6)-- (2.4,1.4);
\draw [line width=2.pt,color=ffffff] (3.,1.8)-- (3.10400838345,1.59198323309);
\draw [line width=2.pt,color=ffffff] (3.4,1.6)-- (3.4,1.4);
\draw (1.33,1.67) node[anchor=north west] {$\Phi(f)$};
\draw (2.33,1.67) node[anchor=north west] {$\Phi(f)$};
\draw (3.33,1.67) node[anchor=north west] {$\Phi(f)$};
\draw (1.9,1.85) node[anchor=north west] {$\Phi(g)$};
\draw (2.9,1.85) node[anchor=north west] {$\Phi(g)$};
\draw (2.89331811076,0.935771686531) node[anchor=north west] {$A$};
\draw (1.35497265409,0.919449718821) node[anchor=north west] {$D_1$};
\draw (2.35877366825,0.919449718821) node[anchor=north west] {$D_2$};
\draw (3.35441369857,0.915369226893) node[anchor=north west] {$D_3$};
\draw (1.35089216216,3.43) node[anchor=north west] {$E_1$};
\draw (2.35877366825,3.43) node[anchor=north west] {$E_2$};
\draw (3.34625271471,3.43) node[anchor=north west] {$E_3$};
\draw (1.7,3.38) node[anchor=north west] {$\Phi(g)(A)$};
\end{tikzpicture}
In the figure $\Phi(f)$ acts vertically and $\Phi(g)$ acts diagonally. \\
Since $\Phi(f)(X\setminus A)\subset \Phi(g)(A)$, we can not use Proposition \ref{prop:ess} directly. 
\label{fig:M1}
\end{figure}

\begin{lem}\label{lem:step 1}
    There are $k\in\N$ with $k\leq 1/\eps$ and two sequences $\{D_i\}_{i\leq k}$ and $\{E_i\}_{i\leq k}$ of measurable subsets of $X$ of positive measure such that 
\begin{enumerate}
\item the $\{D_i\}_{i\leq k}$ are pairwise disjoint as are the $\{E_i\}_{i\leq k}$,
\item $D_1\subset X\setminus A$ and $D_i\subset A$ for $i>1$,
\item $E_k\subset X\setminus \Phi(g)(A)$ and $E_i\subset \Phi(g)(A)$ for $i<k$,
\item $\Phi(f)(D_k)=E_k$ and $E_{k-1}=\Phi(g)(D_k)$.  
\end{enumerate}
\end{lem}
\begin{proof}
Set $B_1:=X\setminus A$ and $C_1:=\Phi(f)(B_1)$. For $i\geq 2$ define recursively 
\[ B_i:=\Phi(g)^{-1}(C_{i-1}\cap \Phi(g)(A)) \text{ and } C_i:=\Phi(f)(B_i).\]

Observe that $\{B_i\}_i$ are pairwise disjoint as are the $\{C_i\}_i$. Suppose now that for $l\geq 1$, we have that $C_i\subset \Phi(g)(A)$ for all $i\leq l$. Since $\Phi(g)$ and $\Phi(f)$ preserve the measure, we have that $\mu(C_i)=\mu(B_1)$ for all $i\leq l$ and hence we have that $l\mu(B_1)\leq 1-\mu(B_1)$. By hypothesis $\mu(B_1)\geq
\eps$, so $l\leq 1/\eps-1$. Therefore there exists $k\leq 1/\eps$, such that $C_k$ is not contained in $\Phi(g)(A)$ and $C_i\subset \Phi(g)(A)$ for every $i<k$. 

Put $E_k:= C_k\setminus\Phi(g)(A)$ and set $D_k:=\Phi(f)^{-1}(E_k)$. Observe that $D_k\subset B_k$ and define recursively $E_i:=\Phi(g)(D_{i+1})$ and $D_i:=\Phi(f)^{-1}(E_i)$. 
\end{proof}

\begin{proof}[Proof of Proposition \ref{prop:bitbig}]
Consider the families $\{D_i\}_{i\leq k}$ and $\{E_i\}_{i\leq k}$ defined in the previous lemma. By Proposition \ref{prop:ess}, there exists a good couple $(A_1,g_1)$ such that $A_1\subset D_1$ and $\Phi(g_1)(A_1)\subset \Phi(f)(A_1)\subset E_1$. For $i\in \{2,\ldots,k\}$, whenever $A_{i-1}$ is defined, we set \[A'_i:=\Phi(g)^{-1}(\Phi(g_{i-1})(A_{i-1}))\subset D_i.\] 
For every $i$ such that $A_i'$ is defined, Proposition \ref{prop:ess} implies that there is a good couple $(A_i,g_i)$ such that $A_i\subset A_i'$ is non-negligible and $\Phi(g_i)(A_i) \subset \Phi(f)(A_i)\subset E_i$. Put $B_k:=A_k$. For $i\in \{1,\ldots,k-1\}$, we define recursively $B_i:=\Phi(g_i)^{-1}(\Phi(g)(B_{i+1}))$. 

Set $A':=A\cup B_1$ and define \[g'(x):=\left\lbrace\begin{array}{ll} g(x)&\text{ if }x\in A\setminus \cup_{i\geq 2} B_i, \\ g_i(x)&\text{ if }x\in B_i.\end{array}\right .\]

By construction, $\Phi(g'):A'\to X$ is injective and preserves the measure. Moreover $(A',g')$ is obtained by \textit{cutting and pasting} $N$-good couples, so it is an $N$-good couple. Let us finally check that $(A,g)\prec (A',g')$. Clearly we have $A'\supset A$ and $\mu(A'\setminus A)=\mu(B_1)>0$. Moreover 
\[ \mu(\lbrace x\in A:\ g(x)\neq g'(x)\rbrace)\leq
  \mu\left(\cup_{i\geq 2} B_i\right)\leq k\mu(B_1)\leq\frac{1}{\eps}(\mu(A')-\mu(A)).\qedhere\]
\end{proof}

\subsubsection*{Step 3}

\noindent We verify now that we can apply Zorn's Lemma to the set of good couples. 

\begin{prop}\label{prop:zornification}
  Every chain for $\prec$ has an upper bound. 
\end{prop}
\begin{proof}
Let us assume for the moment that $\{(A_n,g_n)\}_n$ is a countable chain of good couples. For every $n\in \N$ set \[B_n:=\{x\in A_n:\ g_n(x)=g_{n+1}(x)\},\quad C_n:=\cap_{k\geq n} B_n\quad\text{and}\quad A:=\cup_n C_n.\]

Clearly $A\subset\cup_n A_n$ and we now check that the two measurable subsets have the same measure. In fact, since $\{A_n\}_n$ and $\{C_n\}_n$ are increasing sequences, for every $\eta>0$, there is $K\in\N$ such that \[\mu(\cup_nA_n)-\mu(A_K)<\eta\quad\text{and}\quad \mu(\cup_nC_n)-\mu(C_K)<\eta,\]
hence we have
  \begin{align*}
    \mu(\cup_n A_n)-\mu(A)\leq& 2\eta+\mu(A_K)-\mu(C_K)=2\eta+\mu(A_K\setminus C_K)\\
    =&2\eta +\mu(A_K\cap (\cup_{k\geq K} X\setminus B_k))=2\eta+\mu(\cup_{k\geq K} A_K\setminus B_k) \\
 \leq &2\eta+\sum_{k\geq K}\mu(A_k\setminus B_k)\leq 2\eta+\frac{1}{\eps}\sum_{k\geq K}\mu(A_{k+1}\setminus A_k)\\
 \leq &2\eta+\frac{1}{\eps}\mu(\cup_{k\geq K+1}A_k\setminus A_K)\leq 2\eta+ \frac{\eta}{\eps}. 
  \end{align*}

As $\eta$ is arbitrarily small, we get that $A=\cup_n A_n$ almost everywhere. For $x\in C_n$, observe that $g_n(x)=g_{n+j}(x)$ for every $j\geq 0$. We define \[g(x):=g_n(x)\quad\text{if}\quad x\in C_n.\]
The couple $(A,g)$ is obtained by cutting and pasting $N$-good couples so the couple is $N$-good. Moreover $A\supseteq \cup_n A_n$ almost everywhere and for every $n\in\N$, we have \[\mu(x\in A_n:\ g_n(x)\neq g(x))\leq \mu(A_n\setminus C_n) \leq \frac{1}{\eps} \sum_{k\geq n}\mu(A_{k+1}-A_k)=\frac{1}{\eps}(\mu(A)-\mu(A_n)).\]

 Therefore the couple $(A,g)$ is an upper bound for the countable chain. Consider now an arbitrary chain $\{(A_c,g_c)\}_{c\in C}$ and set $\lambda=\sup_{c\in C} \mu(A_c)$. If there is a good couple $(A_c,g_c)$ such that $\mu(A_c)=\lambda$, then this couple is an upper bound of the chain and there is nothing to prove. Suppose that this is not the case and consider a subsequence $\{(A_n,g_n)\}_{n\in\N}$ of the chain such that $\lim_n \mu(A_n)=\lambda$. Let $(A,g)$ be an upper bound for this sequence. Given any element of the chain $(A_c,g_c)$ there exists $n$ such that $\mu(A_c)\leq \mu(A_n)$ and hence $(A_c,g_c)\prec (A_n,g_n)\prec (A,\ph)$. 
\end{proof}

\subsubsection*{End of the proof of Theorem \ref{thm:fmlsui}}

Let $f\in \widetilde{[\mathcal R_G]}$. By definition of the topology of convergence in measure, a base of neighborhoods of $f$ is given by the open sets \[\mathcal U_{\eps,N}:=\left\lbrace g\in \widetilde{[\mathcal R_G]}:\ \mu\left(\{x\in X:\ g(x)\in Nf(x)\}\right)>1-\eps\right\rbrace,\]
where $\eps>0$ and $N\subset G$ is a neighborhood of the identity. For every neighborhood of the identity $N\subset G$, Proposition \ref{prop:ess} implies that the set of good couples for $f$ is not empty. For $\eps>0$, Proposition \ref{prop:zornification} tells us that there is a maximal good couple $(A,g)$. The maximality of the couple and Proposition \ref{prop:bitbig} imply that $\mu(A)\geq 1-\eps$. So by Lemma \ref{lem: ext countable range} there is $g'\in \widetilde{[\mathcal R_G]}_D$ such that $g'\in \mathcal U_{\eps,N}$.

\section{Topological rank of orbit full groups}

We now use Theorem \ref{thm:densesub} to show that the topological rank of orbit full groups associated to free measure-preserving actions of unimodular locally compact groups is equal to two.

\begin{thm}\label{thm:toporank g delta}
Let $G$ be a locally compact unimodular non-discrete and non-compact Polish group. For every measure-preserving, essentially free and ergodic action of $G$, there is a dense $G_\delta$ of couples $(T,U)$ in $[\mathcal R_G]^2$ which generate a dense free subgroup of $[\mathcal R_G]$ acting freely. In particular, the topological rank of $[\mathcal R_G]$ is $2$. 
\end{thm}
\begin{proof}
Let $G$ be a locally compact unimodular, non-discrete and non-compact Polish group. Suppose that $G$ acts on the probability space $(X,\mu)$ preserving the measure, essentially freely and ergodically. Let us denote by $\free_2$ the free group on two generators and observe that \[\left\lbrace (T,U)\in [\mathcal R_G]^2:\ \overline{\la T,U\ra}=[\mathcal R_{G}]\text{ and }\la T,U\ra\cong \free_2\right\rbrace\] is a $G_\delta$, so we have only to prove that it is dense. 

By Theorem \ref{thm:OEprod}, there exists a (not necessarily free) action of a countable group $\Gamma$ on a measure space $(Y,\nu)$, such that $\mathcal R_G$ is orbit equivalent to  the product action of $\mathbb S^1\times \Gamma$ on $\mathbb S^1\times Y$. Fix a copy of $\Z$ in $\mathbb S^1$ generated by an irrational rotation; then $\Z\times \Gamma$ is dense in $\mathbb S^1\times \Gamma$. By Theorem \ref{thm:densesub}, we have that $[\mathcal R_{\Z\times \Gamma}]$ is dense in $[\mathcal R_G]$. 

The equivalence relation $\mathcal R_{\Z\times \Gamma}$ has cost $1$, by Proposition $V\!I.23$ of \cite{MR1728876} (note that the proof only uses that $\Gamma_1$ acts freely). So we can apply Theorem 1.7 in \cite{lm14nonerg} to get the existence of an aperiodic $T\in [\mathcal R_{\Z\times \Gamma}]$ such that 
\[\left\lbrace U\in [\mathcal R_{\Z\times \Gamma}]: \overline{\la T,U\ra}^{d_u}=[\mathcal R_{\Z\times \Gamma}]\text{ and }\la T,U\ra\cong \free_2\right\rbrace\subset [\mathcal R_{\Z\times \Gamma}]\]
 is a dense subset of $[\mathcal R_{\Z\times\Gamma}]$ with respect to the uniform topology. This concludes the proof since by Theorem 4.4 of \cite{Carderi:2014qr}, the conjugacy class of $T$  is dense in $[\mathcal R_G]$ for the topology of convergence in measure.\end{proof}

\section{The orbit full group as a unitary group}\label{sec:full group unitary group}

In this section, we study the relationship between orbit full groups arising from measure-preserving free actions of locally compact groups and the associated von Neumann algebra. Throughout this section, $G$ will be a locally compact, second-countable unimodular group which we equip with a left and right invariant Haar measure $m$. 

Let us recall the crossed product construction. See the first chapter of \cite{MR500089} for more about this. Note however that our left von Neumann algebra is the right von Neumann algebra in Van Daele's book.  

\begin{df}Let $G$ be a locally compact Polish group.
  For a measure-preserving free action of $G$ on the probability space $(X,\mu)$, the \textbf{crossed product} $\LL^\infty(X,\mu) \rtimes G$ is the von Neumann algebra on $\LL^2(G\times X, m\times \mu)$ generated by 
\begin{itemize} \item the set of unitary operators $\{\lambda_h\times \kappa_h\}_{h\in G}$ where $h\mapsto \kappa_h$ is the Koopman representation of $G$ on $\LL^2(X,\mu)$ and $h\mapsto \lambda_h$ is the left regular representation,
\item the abelian algebra $\LL^\infty(X,\mu)$ which acts on  functions $\xi\in\LL^2(G\times X,\mu)$ by multiplication: for all $f\in\LL^\infty(X,\mu)$, we let $f\xi(g,x)=f(x)\xi(g,x)$.
\end{itemize} 
\end{df}

We will show that this von Neumann algebra is generated by the orbit full group $[\mathcal R_G]$, which can be seen as a unitary group as follows. Recall that since we assume that the action of $G$ is essentially free, the full group $[\mathcal R_G]$ is isomorphic as Polish group to $\widetilde{[\mathcal R_G]}$, as explained in Section \ref{sectfull}.

\begin{df}
Let the full group $\widetilde{[\mathcal R_G]}$ almost-act on $(G\times X, m\times\mu)$ by
  \[g\cdot (h,x)=(g(x) h, g(x) x)\text{ for all } g\in \widetilde{[\mathcal R_G]},\ h\in G, x\in X \]
  and denote by $\pi$ the associated Kooman representation on $\LL^2(G\times X)$.
  That is for every $f\in\LL^2(G\times X)$ and $g\in\widetilde{[\mathcal R_G]}$, we have
  \[\pi(g)\cdot f(h,x)=f(g(x)\inv h, g(x)\inv x).\]
\end{df}

\begin{prop}\label{prop:full group gen vna}
  Let $G$ be a unimodular non-compact, locally compact Polish group acting freely on $(X,\mu)$.  
  \begin{enumerate}[(1)]
  \item The map $\pi$ is a continuous embedding of $[\mathcal R_G]$ into $\mathcal U(M)$.
  \item The full group of $\mathcal R_G$ consists of the intersection of $\Aut(G\times X, m\times \mu)$ with $\mathcal U(G\ltimes\LL^\infty(X))$, seeing both as subgroups of $\mathcal U(\LL^2(G\times X,m\times\mu))$. In particular, it is a closed subgroup of $\mathcal U(\LL^2(G\times X,m\times\mu))$.
  \item The full group generates the von Neumann algebra, that is $\pi([\mathcal R_G])''=\LL^\infty(X,\mu)\rtimes G$.
  \end{enumerate}
\end{prop}
\begin{proof}
  $(1)$ Firstly, observe that the action of $\widetilde{[\mathcal R_G]}$ on $G\times X$ is measure preserving, so that $\pi$ is a unitary representation. 

  Let us now see why $\pi(G)\subseteq \mathcal U(M)$. For this, note that the commutant of $M=\LL^\infty(X)\rtimes G$ is generated by the operators $\tilde f$ for $f\in \LL^\infty(X)'=\LL^\infty(X)$ (acting by $\tilde (\tilde f \xi)(h,x)=f(h\inv x)\xi(h,x))$) and the operators $1\times \rho_g$ where $\rho_g$ is the right regular representation, see \cite[Thm. 3.12]{MR500089}. 
For $g\in[\mathcal R_G]$, we have
\begin{align*}
\pi(g)\tilde f \xi(h,x)&=(\tilde f \xi)(g(x)\inv h,g(x)\inv x)\\
&=f(h\inv x) \xi(g(x)\inv h,g(x)\inv x)\\
&=\tilde f \pi(g) \xi(h,x)
\end{align*}
and we also have for $g'\in G$
\begin{align*}
\pi(g)\rho_{h'} \xi(h,x)&= (\rho_{h'}\xi)(g(x)\inv h,g(x)\inv x)\\
&=\xi(g(x)\inv hh', g(x)\inv x)\\
&=\rho_{h'}\pi(g) \xi(h,x),
\end{align*}
which concludes the proof.

\vspace{0.2cm}

$(2)$ We will prove that every automorphism $T\in \Aut(G\times X)$ which commutes with the operators $\tilde f$ and $\rho_h$ is in the image of the full group.

Fix such a $T$ and put $T(g,x)=(t_1(g,x), t_2(g,x))$. Since $T$ commutes with $\rho_h$, we obtain that
\begin{align*}
 T\rho_h(g,x)&=(t_1(gh,x),t_2(gh,x)\\
&=(t_1(g,x)h,t_2(g,x))\\
\end{align*}
so that $t_2(g,x)$ only depends on $x$, and $t_1(gh,x)=t_1(g,x)h$. If we set $g(x):=t_1(1,x)$ and $t(x):=t_2(1,x)$, then we have $T(h,x)=(g(x)h,t(x))$. Now we observe that since $T$ preserves the measure, $T:X\to X$ also does:
\begin{align*}
m\times\mu(T^{-1}(A\times B))&=m\times\mu(\{(h,x): g(x)h\in A \text{ and } t(x)\in B)\\
&=m\times\mu(\{(h,x): h\in g(x)\inv A \text{ and } x\in t\inv(B))\\
&=\int_X m(g(x)\inv A)\chi_{t\inv(B)}(x)d\mu(x)\\
&=m(A)\int_X\chi_{ t\inv(B)}(x)d\mu(x)\\
&=m(A)\mu( t\inv(B))
\end{align*}

Finally, we exploit the hypothesis that $T$ commutes with the operators $\tilde f$,
\begin{align*}
T^{-1}\tilde f\xi(h,x)=&f(h\inv g(x)^{-1} t(x))\xi(g(x) h,t(x))\\
\tilde f T^{-1}\xi(h,x)=&f(h\inv x)\xi(g(x) h, t(x)).
\end{align*}

Since this is true for every $f$ and $\xi$, we must have that $g(x)x=t(x)$ and hence $T$ is in the image of the full group. 

\vspace{0.2cm}

$(3)$  Since $\lambda\times \kappa(G)$ is already a subgroup of $\pi([\mathcal R_G])$, by definition of the crossed product it suffices to show that $\pi([\mathcal R_G])''$ contains $\LL^\infty(X)$ . For this, it is enough to show that for every $A\subset X$ the characteristic function $\chi_{A}$ belongs to $\pi([\mathcal R_G])''$.

By Theorem \ref{thm:OEprod}, we may assume that $X=\mathbb S_1\times Y$, and that $\RR_G=\RR_{\mathbb S_1\times \Gamma}$, where $\circle\times \Gamma$  acts via a product action. Since $G$ is non compact, $\Gamma$ has infinite orbits, but recall that the action is not necessarily free. 

Let $\RR_\Gamma$ the equivalence relation of the action of $\Gamma$ on $Y$ and $\tilde\RR_\Gamma$ the equivalence relation of the action of $\Gamma$ on $\circle\times Y$ obtained by making $\Gamma$ act trivially on $\circle$. Observe that $\tilde\RR_\Gamma=\RR_\Gamma\times\circle$ as measure spaces and by Theorem \ref{thm:OEprod} (ii), we have a measure preserving isomorphism between $(G\times X,m\times \mu)$ and $(\RR_\Gamma\times\circle\times\circle,\tilde\nu\times L\times L)$. 

By a well-known result of Dye (see e.g. \cite[Thm. 3.5]{MR2583950}), we can choose an aperiodic element $T\in [\tilde\RR_\Gamma]\subset[\RR_{\Gamma\times\circle}]$. Now let $T_{X\setminus A}$ be the first return map induced by $T$ on $X\setminus A$. It is easy to check that the sequence $(T_{X\setminus A}^n)_{n\in\N}$ tends to $\chi_A$ weakly as operators on the Hilbert space $\LL^2( \tilde \RR_\Gamma)$. Since $\mathcal R_{\Gamma\times\circle}=\RR_\Gamma\times \circle\times \circle$ and $\tilde\RR_{\Gamma}=\RR_\Gamma\times\circle$, we deduce that $(T_A^n)_{n\in\N}$ tends to $\chi_A$ weakly as operators on the Hilbert space $\LL^2( \RR_{\Gamma\times\circle})\cong \LL^2(\RR_G)$. Therefore the sequence $(T_{X\setminus A}^n)_{n\in \N}$ tends to $\chi_A$ weakly in $\pi([\RR_G])''$.  
\end{proof}

\section{Extreme amenability of orbit full groups}

Let us recall that a Polish group is \textbf{extremely amenable} if whenever it acts continuously on a compact space, the action has a fixed point. It is \textbf{amenable} if whenever it acts continuously by affine transformations on a compact subset of a locally convex topological vector space, then the action has a fixed point.

The aim of this section is to extend Theorem 5.7 of Giordano and Pestov \cite{MR2311665} to the locally compact setting.

\begin{thm}\label{thm:amenable extreame}
  Let $G$ be a locally compact, non-compact unimodular Polish group. Suppose that $G$ acts freely on the probability space $(X,\mu)$ preserving the probability measure. Then  the following are equivalent.\begin{enumerate}[(i)]
\item $G$ is amenable.
\item $[\mathcal R_G]$ is amenable.
\item $[\mathcal R_G]$ is extremely amenable.
\end{enumerate}
\end{thm}

Before we prove the theorem, let us recall the following useful well-known result which follows from Remark 5.3.29(2) and Corollary 6.2.12 of \cite{MR1799683}.

\begin{thm}\label{thm: injective crossed product and amenability}
    Let $G$ be a locally compact, non-compact unimodular Polish group. Suppose that $G$ acts freely on the probability space $(X,\mu)$ preserving the probability measure. Then $G$ is amenable if and only if the crossed product $\LL^\infty(X,\mu)\rtimes G$ is injective. 
\end{thm}

We will also need the following lemma, which provides basic extremely amenable orbit full groups.

\begin{lem}\label{lem: full groups for compact groups are extremely amenable}Let $G$ be a compact metrisable group acting freely on a standard probability space $(X,\mu)$. Then the associated orbit full group is extremely amenable.
\end{lem}
\begin{proof}As a consequence of \cite[Thm. 3.2]{MR0159923}, we may view $X$ as a Borel $G$-invariant subspace of a compact continuous $G$-space $K$. By \cite[Prop. 3.4.6]{MR2455198}, there is a Borel transversal\footnote{A Borel transversal is a Borel subset which intersects every $G$-orbit at exactly one point.} for the $G$-action on $K$, in particular there is a Borel transversal $Y$ for the $G$-action on $X$. 

Let $\pi$ be the Borel map with takes every $x\in X$ to the only $y\in Y$ such that $y\in G\cdot x$, and equip $Y$ with the pushforward measure $\nu:=\pi_*\mu$. Let $\lambda$ be the Haar probability measure on $G$. By uniqueness of the Haar measure,  the Borel $G$-equivariant bijection
\begin{align*}\Phi: &(G\times Y, \lambda\times \nu)\to (X,\mu)\\
 & (g,y)\mapsto g\cdot y
\end{align*}
is measure-preserving. Moreover, under the identification of $X$ with $G\times Y$, the orbit full group becomes the group $\LL^0(Y,\nu,\Aut(G,\lambda))$ equipped with the topology of convergence in measure. We now have two cases to consider:
\begin{itemize}
\item $G$ is discrete hence finite, in which case $(Y,\nu)$ has to be non-atomic and $\Aut(G,\lambda)$ is a finite permutation group, in particular it is a compact group. Then by a result of Glasner (see \cite[Thm. 4.2.2]{MR2277969}), the group $\LL^0(Y,\nu, \Aut(G,\lambda))$ is extremely amenable.
\item $G$ is non-discrete, in which case $\Aut(G,\lambda)$ is extremely amenable by a result of Giordano and Pestov (see \cite[Thm. 4.5.15]{MR2277969}), which implies that $\LL^0(Y,\nu,\Aut(G,\lambda))$ also is.
\end{itemize}
In either case, we see that the orbit full group $\LL^0(Y,\nu,\Aut(G,\lambda))$ is extremely amenable as desired.
\end{proof}

%

\begin{proof}[Proof of Theorem \ref{thm:amenable extreame}]
Clearly  $(iii)\impl (ii)$, so we will only have to show that $(i)\impl(iii)$ and that $(ii)\impl(i)$.

   $(i)\impl(iii)$:  Suppose the group $G$ is amenable. By Theorem \ref{thm:OEprod}, we can assume that $X$ decomposes as a product $(Y\times \mathbb S^1,\nu\times\lambda)$, and that $\mathcal R_G=\mathcal R\times (\circle\times\circle)$ where $\mathcal R$ is a measure-preserving countable aperiodic amenable equivalence relation. By Connes-Feldman-Weiss' theorem \cite{MR662736}, we can actually assume that $\mathcal R=\mathcal R_\Gamma$ where $\Gamma:=\bigoplus_{n\in\N}\Z/2\Z$ is acting freeely on $(Y,\nu)$.
   
Let then $H:=\Gamma\times\mathbb S^1$, then we have a natural $H$-action on $Y\times \mathbb S^1$ which induces the same equivalence relation as $G$. We thus only have to show that $[\mathcal R_H]$ is extremely amenable. The group $H$ is naturally written as an increasing union of compact groups $K_n:=(\bigoplus_{k\leq n}\Z/2\Z)\times\circle$.

Note that the reunion $\cup_n\widetilde{[K_n]}_D$ is dense in $\widetilde{[H]}_D$. In fact given for every $f\in \widetilde{[H]}_D$ and every $\eps$, there exists $N>0$ such that \[A:=\{x\in X:\ f(x)\in K_N\}\] has measure bigger than $1-\eps$. Therefore by Lemma \ref{lem: ext countable range} we can extend $f_A$, the restriction of $f$ to $A$, to an element of $\widetilde{[K_N]}_D$ which is closed to $f$. 

By Theorem \ref{thm:fmlsui} the group $[H]_D$ is dense in $[\mathcal R_H]$, so the reunion $\cup_n[K_n]$ is dense in $[\mathcal R_H]$. Observe now that $\cup_n[K_n]_D\subseteq \bigcup_n[\mathcal R_{K_n}]$ and hence $\bigcup_n[\mathcal R_{K_n}]$ is dense in $[\mathcal R_H]$. Finally observe that the full groups $[\RR_{K_n}]$ are extremely amenable by Lemma \ref{lem: full groups for compact groups are extremely amenable}. So $[\RR_H]$ contains an increasing sequence of extremely amenable subgroups, whose union is dense, therefore $[\RR_H]=[\mathcal R_G]$ is extremely amenable. 


  \vspace{0.2cm}

$(ii)\impl(i)$:  Suppose the full group $[\mathcal R_G]$ is amenable. By Theorem \ref{thm: injective crossed product and amenability}, the amenability of $G$ is equivalent to the injectivity of the crossed product $\LL^\infty(X,\mu)\rtimes G$. Moreover by the celebrated result of Connes \cite[Thm. 6]{MR0454659}, the injectivity of a von Neumann algebra $M\subseteq \BH$ is equivalent to Schwartz's property (P), which means that whenever $x\in \BH$, the  closed convex hull of $(uxu^*)_{u\in\mathcal U(M)}$ intersects the commutant of $M$. It thus suffices to prove that $\LL^\infty(X)\rtimes G$ has Schwartz's property (P).

To this end, let $x\in \BH$. Then the convex closed hull $K$ of $(uxu^*)_{u\in\mathcal U(\LL^\infty(X,\mu)\rtimes G)}$ is a weakly compact convex set onto which $[\RR_G]$ acts continuously by conjugation. Since $[\RR_G]$ is amenable and this action is the restriction of a linear hence affine action, there exists $x_0\in K$ which is fixed by the conjugation action. This means that $x_0\in\pi([\RR_G])'=(\pi([\RR_G])'')'$ so $x_0$ belongs to the commutant of $\LL^\infty(X)\rtimes G$ by item (3) of Proposition \ref{prop:full group gen vna}, which concludes the proof.
\end{proof}

The proof of $(ii)\impl(i)$ actually shows that the von Neumann algebra generated by an amenable unitary group is injective. Note that de la Harpe proved that a von Neumann algebra is injective if and only if its unitary group is amenable \cite{MR548116}. Our proof of $(ii)\impl(i)$ is essentially a reformulation of his.

\appendix
\section{Haar measures for equivalence relations}

The content of this appendix is standard and can be carried out in a much more general setting (see \cite{MR1799683}). 
However, extracting the statements we need can be difficult, so we give complete proofs for which we claim no originality.

\subsection{Invariant Haar systems}

When $G=\Gamma$ is a discrete group, the first-coordinate projection $\pi:\mathcal R_\Gamma\to X$ has countable fibers, which allows us to define a \textbf{Haar measure} $M$ on $\mathcal R_\Gamma$ by integrating the counting measure over the fibers: for all Borel $A\subseteq\mathcal R_\Gamma$, $$M(A)=\int_X\abs{\pi\inv(\{x\})\cap A}d\mu(x).$$

The definition of a Haar measure in a more general context of locally compact groups is however more complicated. 

\begin{df}Let $\mathcal R$ be a Borel equivalence relation on $(X,\mu)$. An \textbf{invariant Haar system} on $\mathcal R$ is a family $(m_x)_{x\in X}$ of Borel measures on $X$ which satisfy the following properties:
\begin{enumerate}[(1)]
\item \label{condition: invariance}(\textit{invariance}) There is a full measure subset $X'$ of $X$ such that for all $(x,y)\in\mathcal R\cap(X'\times X')$, $m_x=m_y$. 
\item For all $x\in X$, $m_x$ is non trivial and supported on $[x]_{\mathcal R}$ (i.e. $m_x(X\setminus[x]_{\mathcal R})=0$ and $m_x([x]_\mathcal R)>0)$. 
\item (\textit{measurability}) For all Borel $A\subseteq \mathcal R$, the map $x\mapsto m_x(A_x)$ is Borel, where $A_x:=\{y\in X: (x,y)\in A\}$.
\item \label{cond:sigma finite}($\sigma$-\textit{finiteness}) There exists an exhausting increasing sequence of Borel subsets $(A_n)$ of $\mathcal R$ such that for all $n\in\N$, one has $\int_Xm_x((A_n)_x)d\mu(x)<+\infty$. 

\item \label{cond:haar full measure}For all full measure subsets $X'$ of $X$, one has  $m_x(X\setminus X')=0$ for $\mu$-almost all $x\in X$.
\end{enumerate}
\end{df}
\begin{rmq}Note that condition (\ref{cond:haar full measure}) allows one to transport a Haar system on $\mathcal R$ to a Haar system on $\mathcal R'$ whenever $\mathcal R$ and $\mathcal R'$ are orbit equivalent.  
\end{rmq}
\begin{exemple}
Suppose $\mathcal R$ is a Borel countable non-singular equivalence relation on $(X,\mu)$. Then an invariant Haar system on $\mathcal R$ is given by letting $m_x$ be the counting measure on $[x]_{\mathcal R}$. 
\end{exemple}
\begin{exemple}\label{ex: haar system from right haar on lc group}
Suppose $G$ is a locally compact Polish group with right Haar measure $\lambda$. Then given an essentially free measure-preserving $G$-action on $(X,\mu)$, one can endow $\mathcal R_G$ with an invariant Haar system $(\lambda_x)_{x\in X}$ given by the natural identification $g\mapsto g\cdot x$ between $(G,\lambda)$ and $[x]_{\mathcal R}$. In other words, given a Borel subset $A$ of $G$ and $x\in X$, we set $\lambda_x(A\cdot x):=\lambda(A)$. Note that such an identification only makes sense when $x$ belongs to the free part of the action, so when $x$ does not belong to it we define $\lambda_x$ to be the Dirac measure on $x$.

Let us check that the field of measure $(\lambda_x)_{x\in X}$ is an invariant Haar system. For a Borel subset $A$ of $G$, we have 
$$\lambda_{gx}(A\cdot x)=\lambda_{gx}(Ag\inv g\cdot x)=\lambda(Ag\inv)=\lambda(A)=m_x(A\cdot x),$$
so condition (1) is satisfied. One can easily check that conditions (2), (3) and (4) are satisfied, while (\ref{cond:haar full measure})  is a consequence of the Fubini theorem and the fact that the $G$-action preserves the measure: if $X'$ has full measure in $X$ then for almost all $x\in X$, for $\lambda$-almost $g\in G$ one has $g\cdot x\in X'$. 

Note that when $G$ is discrete, this definition of the Haar measure coincides with the previous one.
\end{exemple}

\begin{rmq}
  Actually, as the expert reader knows, one can define a Haar measure on $\mathcal R_G$ regardless of the freeness of the $G$-action, whenever $G$ is locally compact Polish. But since the construction of the measure is significantly more complicated, and since we will only deal with non-free actions when $G$ is discrete, we chose not to present this more general setting.
\end{rmq}

\begin{exemple}If $(\mathcal R_1,(m^1_x)_{x\in X})$ and $(\mathcal R_2,(m^2_y)_{y\in Y})$ are measured equivalence relations on $(X,\mu)$ and $(Y,\nu)$ respectively, then $(\mathcal R_1\times\mathcal R_2,(m^1_x\times m^2_y)_{(x,y)\in X\times Y})$ is a measured equivalence relation on $(X\times Y, \mu\times \nu)$. A particular case of interest to us is when $\mathcal R_1$ is the transitive equivalence relation and $\mathcal R_2$ is a countable measure-preserving equivalence relation. Indeed by Theorem \ref{thm:OEprod} every measured equivalence relation arising from a free action of a non discrete unimodular locally compact group is of this form. 
\end{exemple}

\begin{rmq}There can be a lot of different invariant Haar systems on an equivalence relation $\mathcal R$, even in the ergodic case. For instance, if $\mathcal R$ is the transitive equivalence relation on $(X,\mu)$, then any choice of Borel $\sigma$-finite measure $\nu$ on $X$ which is absolutely continuous with respect to $\mu$ yields an invariant Haar system $(m_x)_{x\in X}$ given by $m_x=\nu$. In the next section, we will add a condition which yields uniqueness: unimodularity.
\end{rmq}

By Weil's theorem,
a Polish group which admits a right-invariant measure is locally compact. Similarly, the existence of an invariant Haar system on an equivalence relation forces the acting group to be locally compact. Let a Polish group $G$ act freely on $(X,\mu)$, and suppose that there exists an invariant Haar system $(m_x)_{x\in X}$ on $\mathcal R_G$. Then we can define a natural right-invariant measure on $G$ as \[ \lambda(A):=\int_X m_x(A\cdot x)\d\mu(x),\ \text{for}\ A\subset G.\]
This measure is not always $\sigma$-finite\footnote{Let $\Gamma=\Z/2\Z$ act on $[0,1]$ via $T:x\mapsto (1-x)$ and take a $T$-invariant function $f:[0,1]\to [0,+\infty[$ which is not integrable, then $m_x=f(x)(\delta_x+\delta_{T(x)})$ is an invariant Haar system but the associated measure on $\Z/2\Z$ is infinite. }, but we now show how this can be circumvented.


\begin{thm}\label{thm: lc is closed under OE}Let $G$ be a Polish group acting freely on $(X,\mu)$ in a measure-preserving manner. If $\mathcal R_G$ has an invariant Haar system, then $G$ is locally compact.
\end{thm}
\begin{proof}
We will show that there exists a non-trivial right-\textit{quasi}-invariant Borel probability measure on $G$. This implies that $G$ is locally compact by Mackey's theorem \cite[Thm. 7.1]{MR0089999}. 

Let $(m_x)$ be an invariant Haar system on $\mathcal R_G$, let $(A_n)$ be a partition of $\mathcal R_G$ into Borel sets of finite measure. We define a new Haar system $(\eta_x)$ of \textit{probability measures} on $\mathcal R_G$ by putting, for every $x\in X$ and Borel $A\subseteq X$, 
\[\eta_x(A):=\sum_{n=0}^{\infty}\frac 1{2^{n+1}}\frac{m_x((A\cap A_n)_x)}{m_x((A_n)_x)}\]
Then our Haar system satisfies all the axioms of invariant Haar systems except of course invariance (condition (\ref{condition: invariance})) , which can be replaced by
\begin{enumerate}[(1')]
\item (\textit{quasi-invariance}) There is a full measure subset $X'$ of $X$ such that for all $(x,y)\in \mathcal R\cap (X'\times X')$, $[\eta_x]=[\eta_y]$.
\end{enumerate}
As before, we can integrate the Haar system to obtain a probability measure on $G$: \[ \lambda(A):=\int_X \eta_x(A\cdot x)\d\mu(x),\ \text{for}\ A\subset G.\]

To complete the proof, we will show that $\lambda$ is quasi-invariant with respect to the right multiplication. For this, suppose that $\lambda(A)=0$, then by definition for almost all $x\in X$ one has $\eta_x(A\cdot x)=0$ which implies by $(1')$ that for every $g\in G$ and almost all $x\in X$, $\eta_{gx}(A\cdot x)=0$. Since moreover we have that $g_*\mu=\mu$, we can conclude the proof:
\[\lambda(Ag)=\int_X\eta_x(Agx)d\mu(x)=\int_X\eta_{g^{-1}x}(Ax)d\mu(x)=0.\qedhere\]
\end{proof}

\subsection{Unimodularity}

For a measured equivalence relation $\mathcal R$ on $(X,\mu)$, the \textbf{pre-orbit full group} $[\mathcal R]_B$ is the group of all Borel bijections $T:X\to X$ which preserve $\mu$, and such that for all $x\in X$, one has $(x,T(x))\in\mathcal R$. The pre-orbit full group has two natural actions on $\mathcal R$:
\begin{itemize}
\item the \textbf{left action} defined by $l_T(x,y)=(T(x),y)$ for all $(x,y)\in\mathcal R$ and
\item the \textbf{right action} defined by $r_T(x,y)=(x,T(y))$ for all $(x,y)\in\mathcal R$.
\end{itemize}
These two actions are conjugated by the \textbf{flip} $\sigma$ defined by $\sigma(x,y):=(y,x)$. 

For $A\subset \mathcal R$, as in the last section, we put $A_x=\{y\in X: (x,y)\in A\}$. 
Every invariant Haar system $(m_x)_{x\in X}$ allows us to equip $\mathcal R$ with a natural measure $M$ defined as follows
$$M(A):=\int_Xm_x(A_x)d\mu(x)\quad \text{for every } A\subseteq\RR \text{ Borel}.$$
Note that condition (\ref{cond:sigma finite}) on $(m_x)$ corresponds to the $\sigma$-finiteness of $(\mathcal R,M)$.

\begin{lem}
  The left action of the pre-full group on $\mathcal R$ preserves $M$. 
\end{lem}
\begin{proof}
For all Borel $A\subseteq X$ and all $T\in[\mathcal R]_B$, one has
\begin{align*}
M(l_TA)&=\int_Xm_x((l_TA)_x)d\mu(x)\\
&=\int_Xm_{T\inv (x)}((l_TA)_{x})d\mu(x)\\
&=\int_Xm_{T\inv (x)}((A)_{T\inv(x)})d\mu(x)\\
&=\int_Xm_x(A_x)d\mu(x)=M(A),\end{align*}
so the measure $M$ is preserved by the left action of the pre-orbit full group.
\end{proof}

Denote by $\Aut(\mathcal R,M)$ the group of measure-preserving Borel bijections of $\mathcal R$, two such bijections being identified up to measure zero. Then the left action defines a morphism $[\mathcal R ]_B\to\Aut(\mathcal R,M)$ which factors through the orbit full group $[\mathcal R]$. So the orbit full group (pre-) acts in a measure-preserving manner on $(\mathcal R,M)$. 

\begin{df}An invariant Haar system $(m_x)$ on a Borel equivalence relation $\mathcal R$ is called \textbf{unimodular} if the flip preserves $M$.
\end{df}

As the name suggests, free actions of unimodular locally compact groups give rise to unimodular Haar systems. 

\begin{prop}Let $G$ be a unimodular locally compact group acting essentially freely on $(X,\mu)$ and let $\lambda$ be a Haar measure on $G$. Then the associated invariant Haar system $(\lambda_x)_{x\in X}$ on $\mathcal R_G$ given by Example \ref{ex: haar system from right haar on lc group} is unimodular.
\end{prop}
\begin{proof}
By Lemma \ref{lem:free part is Borel} we may assume that $G$ acts freely. Let $\Phi:X\times G\to \mathcal R_G$ the Borel identification given by $\Phi(x,g):=(x,g\cdot x)$. By definition, the measure $M$ on $\mathcal R_G$ obtained by $M(A)=\int_X\lambda_x(A_x)d\mu(x)$ is just the product measure, $\Phi_*(\mu\otimes\lambda)=M$. Therefore in order to show that $(\lambda_x)_x$ is unimodular, we need to show that the map 
\begin{align*}
\Psi:=&\Phi\circ\sigma\circ \Phi\inv: (X\times G, \mu\otimes \lambda)\to (X\times G,\mu\otimes \lambda)
\end{align*}
is measure preserving. Observe that $\Psi(x,g)=(gx,g^{-1})$. For a set $C$, let $\chi_C$ denote its characteristic function. Let $A\subseteq X$ and $B\subseteq G$ be Borel sets, then we have
\begin{align*}
\Psi_*(\mu\otimes\lambda)(A\times B)&=\int_{X\times G}\chi_{A\times B}(g\cdot x,g\inv)d\mu\otimes\lambda(g,x)\\
&=\int_{X\times G}\chi_{A}(g\cdot x)\chi_B(g\inv)d\mu\otimes\lambda(g,x)\\
&=\int_G\chi_B(g\inv)\left(\int_X\chi_A(g\cdot x)d\mu(x)\right)d\lambda(g)\\
&=\int_G\chi_B(g\inv)\mu(A)d\lambda(g)\\
&=\lambda(B)\mu(A),
\end{align*}
where the last three equalities are respectively consequences of Fubini's theorem, the fact that $G$ preserves the measure and  the unimodularity of $G$. By uniqueness of the product measure, we conclude that $\Psi_*(\mu\times\lambda)=\mu\times\lambda$ as desired. 
\end{proof}

\begin{rmq}
Let $G$ be a unimodular locally compact Polish group acting essentially freely on $(X,\mu)$, let $\lambda$ be a Haar measure on $G$ and let $(\lambda_x)$ be the associated unimodular invariant Haar system on $\mathcal R_G$. Then by the above proposition the right $[\mathcal R_G]$-action on $(\mathcal R_G,M)$ gives an embedding $[\mathcal R_G]\into \LL^0(X,\mu,\Aut(G,\lambda))$ (in other words, the full group acts on every $\mathcal R_G$-class in a measure-preserving manner.). In particular, if $(Y,\nu)$ is a standard $\sigma$-finite space, then the group $\LL^0(X,\mu,\Aut(Y,\nu))$ contains as a closed subgroup every orbit full group arising from a measure-preserving free action of a non-discrete unimodular Polish locally compact group. For a similar statement in the discrete case, see  \cite[Prop. 13]{KLMamplegen}.
\end{rmq}

\begin{thm}\label{thm:haaruni}
  Let $G$ be a Polish group acting freely on $(X,\mu)$. If there is a unimodular invariant Haar system $(m_x)_{x\in X}$ on $\mathcal R_G$, then $G$ is locally compact unimodular.
  
If the action is moreover ergodic, then there exists a constant $c>0$ such that for almost all $x\in X$, one has $m_x=c\lambda_x$, where $\lambda_x$ is the invariant Haar system associated to a fixed Haar measure $\lambda$ on  $G$.
\end{thm}
\begin{proof}
First note that by Theorem \ref{thm: lc is closed under OE}, $G$ has to be locally compact. We fix a left-invariant Borel probability measure $\lambda$ on $G$. 
For every $x$ in the free part of the action, consider the $G$-equivariant bijective Borel map $\phi_x:[x]_{\mathcal R}\to G$ defined by $\phi_x(y)\cdot x=y$; then the pushforward measure $\eta_x:=(\phi_{x})_*m_x$ is a $\sigma$-finite measure on $G$. 

Since the right action of the orbit full group $[\mathcal R]$ on $(\mathcal R,M)$ is conjugate to the left action by the flip, unimodularity yields that the right action of $[\mathcal R]$ on $\mathcal R$ preserves $M$. In particular, the right action of $G$ on $\mathcal R$ preserves $M$, so for a fixed $g\in G$, and any $A\subseteq\mathcal R$ we have
\begin{align*}
\int_Xm_x(A_x)d\mu(x)=\int_Xm_x((r(g)A)_x)d\mu(x)=\int_Xm_x(gA_x)d\mu(x).
\end{align*}
By the uniqueness of disintegration, this implies that for almost all $x\in X$, $g_*m_x=m_x$. Then by Fubini's theorem, for almost all $x\in X$ and $\lambda$-almost all $g\in G$, $g_*m_x=m_x$. Since $\phi_x:[x]_\mathcal R\to G$ is left $G$-equivariant, this implies that for almost all $x\in X$, there is a full measure subgroup of $G$ which preserves $\eta_x$ when acting on the left.  But every full measure subgroup of $G$ equates $G$ (see e.g. \cite[Prop. B.1]{MR776417}),  so for almost all $x\in X$, one has that $\eta_x$ is a Borel $\sigma$-finite left-invariant measure on $G$. By uniqueness of the Haar measure, we conclude that for almost all $x\in X$, the measure $\eta_x$ is a multiple of $\lambda$. 

Fix a Borel subset $K$ of $G$ such that $\lambda(K)=1$. For all $x\in X$, we let $c_x=\eta_x(K)$, then $x\mapsto c_x$ is Borel and we have $\eta_x=c_x\lambda$. Moreover for all $(g,x)\in G\times X$,
\[c_{gx}=\eta_{gx}(K)=m_{gx}(Kgx)=m_{x}(Kgx)=\eta_x(Kg),\]
so that $c_{gx}=\Delta(g)c_x$, where $\Delta$ is the modular function on $G$.


Let $g\in G$, let $a>0$ be an essential value of the function $x\mapsto c_x$ and consider the set of positive measure $A:=\{x\in X: a/2<c_x<3a/2\}$. By Poincaré's recurrence theorem for almost all $x\in A$ there is an infinite subset $S_x\subset \N$ such that $g^kx\in A$ for every $k\in S_x$. So for $x\in A$ we have that \[a/2<\Delta(g^k)c_x=\Delta(g)^kc_x<3a/2\quad \text{ for all }k\in S_x,\] which implies that $\Delta(g)=1$, and we conclude that $G$ is unimodular. 

Therefore $c_{gx}=c_x$ and the function $x\mapsto c_x$ is $G$-invariant. So whenever the $G$-action is ergodic, $c_x$ is a.s. constant, which yields the second part of the theorem. 


\end{proof}

Let us point out that when the acting group $G$ is already known to be locally compact, the freeness hypothesis above can be replaced by almost freeness, since we know by Lemma \ref{lem:free part is Borel} that the free part of the action is a Borel set and invariant Haar systems restrict well to full measure Borel subsets. So unimodular locally compact groups form a closed class under orbit equivalence among locally compact groups.

\bibliographystyle{alpha}
\bibliography{/Users/francoislemaitre/Dropbox/Maths/biblio}

\begin{thebibliography}{GTW05}

\bibitem[ADR00]{MR1799683}
C.~Anantharaman-Delaroche and J.~Renault.
\newblock {\em Amenable groupoids}, volume~36 of {\em Monographies de
  L'Enseignement Math{\'e}matique [Monographs of L'Enseignement
  Math{\'e}matique]}.
\newblock L'Enseignement Math{\'e}matique, Geneva, 2000.
\newblock With a foreword by Georges Skandalis and Appendix B by E. Germain.

\bibitem[AEG94]{MR1250814}
Scot Adams, George~A. Elliott, and Thierry Giordano.
\newblock Amenable actions of groups.
\newblock {\em Trans. Amer. Math. Soc.}, 344(2):803--822, 1994.

\bibitem[Bec13]{MR2995370}
Howard Becker.
\newblock Cocycles and continuity.
\newblock {\em Trans. Amer. Math. Soc.}, 365(2):671--719, 2013.

\bibitem[BK96]{MR1425877}
Howard Becker and Alexander~S. Kechris.
\newblock {\em The descriptive set theory of {P}olish group actions}, volume
  232 of {\em London Mathematical Society Lecture Note Series}.
\newblock Cambridge University Press, Cambridge, 1996.

\bibitem[CFW81]{MR662736}
A.~Connes, J.~Feldman, and B.~Weiss.
\newblock An amenable equivalence relation is generated by a single
  transformation.
\newblock {\em Ergod. Th. \& Dynam. Sys.}, 1(4):431--450 (1982), 1981.

\bibitem[CM16]{Carderi:2014qr}
Alessandro Carderi and Fran{\c c}ois~Le Ma{\^\i}tre.
\newblock More {P}olish full groups.
\newblock {\em Topol. Appl.}, 202:80--105, 2016.

\bibitem[Con76]{MR0454659}
A.~Connes.
\newblock Classification of injective factors. {C}ases {$II_{1},$} {$II_{\infty
  },$} {$III_{\lambda },$} {$\lambda \not=1$}.
\newblock {\em Ann. of Math. (2)}, 104(1):73--115, 1976.

\bibitem[Dan00]{MR1760945}
Alexandre~I. Danilenko.
\newblock Point realization of {B}oolean actions of countable inductive limits
  of locally compact groups.
\newblock {\em Mat. Fiz. Anal. Geom.}, 7(1):35--48, 2000.

\bibitem[dlH79]{MR548116}
P.~de~la Harpe.
\newblock Moyennabilit{\'e} du groupe unitaire et propri{\'e}t{\'e} {$P$} de
  {S}chwartz des alg{\`e}bres de von {N}eumann.
\newblock In {\em Alg{\`e}bres d'op{\'e}rateurs ({S}{\'e}m., {L}es
  {P}lans-sur-{B}ex, 1978)}, volume 725 of {\em Lecture Notes in Math.}, pages
  220--227. Springer, Berlin, 1979.

\bibitem[Dye59]{MR0131516}
H.~A. Dye.
\newblock On groups of measure preserving transformation. {I}.
\newblock {\em Amer. J. Math.}, 81:119--159, 1959.

\bibitem[FM77]{MR0578656}
Jacob Feldman and Calvin~C. Moore.
\newblock Ergodic equivalence relations, cohomology, and von {N}eumann
  algebras. {I}.
\newblock {\em Trans. Amer. Math. Soc.}, 234(2):289--324, 1977.

\bibitem[For74]{MR0417388}
Peter Forrest.
\newblock On the virtual groups defined by ergodic actions of {$R^{n}$} and
  {${\bf Z}^{n}$}.
\newblock {\em Advances in Math.}, 14:271--308, 1974.

\bibitem[Fur99a]{MR1740986}
Alex Furman.
\newblock Gromov's measure equivalence and rigidity of higher rank lattices.
\newblock {\em Ann. of Math. (2)}, 150(3):1059--1081, 1999.

\bibitem[Fur99b]{MR1740985}
Alex Furman.
\newblock Orbit equivalence rigidity.
\newblock {\em Ann. of Math. (2)}, 150(3):1083--1108, 1999.

\bibitem[Gab00]{MR1728876}
Damien Gaboriau.
\newblock Co\^ut des relations d'{\'e}quivalence et des groupes.
\newblock {\em Invent. Math.}, 139(1):41--98, 2000.

\bibitem[Gab05]{MR2183295}
D.~Gaboriau.
\newblock Examples of groups that are measure equivalent to the free group.
\newblock {\em Ergod. Th. \& Dynam. Sys.}, 25(6):1809--1827, 2005.

\bibitem[Gao09]{MR2455198}
Su~Gao.
\newblock {\em Invariant descriptive set theory}, volume 293 of {\em Pure and
  Applied Mathematics (Boca Raton)}.
\newblock CRC Press, Boca Raton, FL, 2009.

\bibitem[GM83]{MR708367}
M.~Gromov and V.~D. Milman.
\newblock A topological application of the isoperimetric inequality.
\newblock {\em Amer. J. Math.}, 105(4):843--854, 1983.

\bibitem[GP02]{MR1891002}
Thierry Giordano and Vladimir Pestov.
\newblock Some extremely amenable groups.
\newblock {\em C. R. Math. Acad. Sci. Paris}, 334(4):273--278, 2002.

\bibitem[GP07]{MR2311665}
Thierry Giordano and Vladimir Pestov.
\newblock Some extremely amenable groups related to operator algebras and
  ergodic theory.
\newblock {\em J. Inst. Math. Jussieu}, 6(2):279--315, 2007.

\bibitem[GTW05]{MR2191233}
E.~Glasner, B.~Tsirelson, and B.~Weiss.
\newblock The automorphism group of the {G}aussian measure cannot act
  pointwise.
\newblock {\em Israel J. Math.}, 148:305--329, 2005.

\bibitem[HC75]{MR0412369}
Wojchiech Herer and Jens Peter~Reus Christensen.
\newblock On the existence of pathological submeasures and the construction of
  exotic topological groups.
\newblock {\em Math. Ann.}, 213:203--210, 1975.

\bibitem[Hud93]{MR1207735}
Steven~M. Hudson.
\newblock A marriage theorem with {L}ebesgue measure.
\newblock {\em J. Combin. Theory Ser. A}, 62(2):234--251, 1993.

\bibitem[Kec95]{MR1321597}
Alexander~S. Kechris.
\newblock {\em Classical descriptive set theory}, volume 156 of {\em Graduate
  Texts in Mathematics}.
\newblock Springer-Verlag, New York, 1995.

\bibitem[Kec10]{MR2583950}
Alexander~S. Kechris.
\newblock {\em Global aspects of ergodic group actions}, volume 160 of {\em
  Mathematical Surveys and Monographs}.
\newblock American Mathematical Society, Providence, RI, 2010.

\bibitem[KM15]{KLMamplegen}
Adriane Ka{\"\i}chouh and Fran{\c c}ois~Le Ma{\^\i}tre.
\newblock Connected {P}olish groups with ample generics.
\newblock {\em Bull. London Math. Soc.}, 47(6):996--1009, 2015.

\bibitem[KPV15]{Kyed:2013lq}
David Kyed, Henrik~Densing Petersen, and Stefaan Vaes.
\newblock {$L^2$}-{B}etti numbers of locally compact groups and their cross
  section equivalence relations.
\newblock {\em Trans. Amer. Math. Soc.}, 367(7):4917--4956, 2015.

\bibitem[LM14]{gentopergo}
Fran{\c c}ois Le~Ma{\^\i}tre.
\newblock The number of topological generators for full groups of ergodic
  equivalence relations.
\newblock {\em Invent. Math.}, 198:261--268, 2014.

\bibitem[LM15]{lm14nonerg}
Fran{\c c}ois Le~Ma{\^\i}tre.
\newblock On full groups of non-ergodic probability measure preserving
  equivalence relations.
\newblock {\em to appear in Ergodic Theory Dyn. Syst.}, 2015.

\bibitem[Mac57]{MR0089999}
George~W. Mackey.
\newblock Borel structure in groups and their duals.
\newblock {\em Trans. Amer. Math. Soc.}, 85:134--165, 1957.

\bibitem[Mac62]{MR0143874}
George~W. Mackey.
\newblock Point realizations of transformation groups.
\newblock {\em Illinois J. Math.}, 6:327--335, 1962.

\bibitem[MRV13]{zbMATH06244595}
Niels {Meesschaert}, Sven {Raum}, and Stefaan {Vaes}.
\newblock {Stable orbit equivalence of Bernoulli actions of free groups and
  isomorphism of some of their factor actions.}
\newblock {\em {Expo. Math.}}, 31(3):274--294, 2013.

\bibitem[OW87]{MR910005}
Donald~S. Ornstein and Benjamin Weiss.
\newblock Entropy and isomorphism theorems for actions of amenable groups.
\newblock {\em J. Analyse Math.}, 48:1--141, 1987.

\bibitem[Pes06]{MR2277969}
Vladimir Pestov.
\newblock {\em Dynamics of infinite-dimensional groups}, volume~40 of {\em
  University Lecture Series}.
\newblock American Mathematical Society, Providence, RI, 2006.
\newblock The Ramsey-Dvoretzky-Milman phenomenon, Revised edition of {{\i}t
  Dynamics of infinite-dimensional groups and Ramsey-type phenomena} [Inst.
  Mat. Pura. Apl. (IMPA), Rio de Janeiro, 2005; MR2164572].

\bibitem[Pet13]{MR3072156}
Henrik~Densing Petersen.
\newblock {$L^2$}-{B}etti numbers of locally compact groups.
\newblock {\em C. R. Math. Acad. Sci. Paris}, 351(9-10):339--342, 2013.

\bibitem[Pra81]{MR624915}
V.~S. Prasad.
\newblock Generating dense subgroups of measure preserving transformations.
\newblock {\em Proc. Amer. Math. Soc.}, 83(2):286--288, 1981.

\bibitem[T{\"o}r06]{MR2210067}
Asger T{\"o}rnquist.
\newblock Orbit equivalence and actions of {$\Bbb F_n$}.
\newblock {\em J. Symbolic Logic}, 71(1):265--282, 2006.

\bibitem[Var63]{MR0159923}
V.~S. Varadarajan.
\newblock Groups of automorphisms of {B}orel spaces.
\newblock {\em Trans. Amer. Math. Soc.}, 109:191--220, 1963.

\bibitem[vD78]{MR500089}
A.~van Daele.
\newblock {\em Continuous crossed products and type {${\rm III}$} von {N}eumann
  algebras}, volume~31 of {\em London Mathematical Society Lecture Note
  Series}.
\newblock Cambridge University Press, Cambridge, 1978.

\bibitem[Zim84]{MR776417}
Robert~J. Zimmer.
\newblock {\em Ergodic theory and semisimple groups}, volume~81 of {\em
  Monographs in Mathematics}.
\newblock Birkh{\"a}user Verlag, Basel, 1984.

\end{thebibliography}

\end{document}